\documentclass[preprint,3p]{elsarticle}
\usepackage{geometry,graphicx,xcolor,amsmath,amssymb,bm, amsthm,mathtools,stanli,pgfplots,subcaption,natbib}
\usepackage[colorlinks]{hyperref}
\pgfplotsset{compat=1.18}

\title{Minimum-weight frames under subresonant harmonic excitation: Robust constraints on dynamic compliance, peak input power, and eigenfrequencies}
\author{Marek Tyburec}
\author{Marouan Handa}
\author{Jan Havelka}
\author{Jan Zeman}
\journal{}
\newtheorem{lemma}{Lemma}
\newtheorem{proposition}{Proposition}
\newtheorem{definition}{Definition}

\DeclareMathOperator*{\argmin}{arg\,min}

\newcommand{\T}{^{\mathsf T}}
\newcommand{\TC}{^{\mathsf H}}
\newcommand{\nel}{n_\mathrm{e}}
\newcommand{\ndof}{n_\mathrm{dof}}
\newcommand{\dK}{d_\mathrm{K}}
\newcommand{\A}{\mathrm{A}}
\newcommand{\R}{\mathrm{R}}
\newcommand{\Real}{\mathrm{Re}}
\newcommand{\I}{\mathrm{Im}}
\newcommand{\range}[1]{\ensuremath{\mathcal{R}\!\left(#1\right)}}
\newcommand{\kernel}[1]{\ensuremath{\mathcal{N}\!\left(#1\right)}}

\numberwithin{equation}{section}

\begin{document}

\begin{abstract}
This work addresses minimum-weight design of undamped Euler--Bernoulli frame structures under subresonant single-frequency harmonic excitations, focusing on (robust) dynamic compliance and (robust) peak input power with ellipsoidal load uncertainty. We develop a semidefinite reformulation of robust dynamic compliance for subresonant single‑frequency excitation and prove its equivalence to robust peak input power. We show that both these response measures admit an exact reformulation as a free-vibration eigenvalue constraint with design-independent mass augmentation, unifying static, dynamic, and modal requirements. Despite the nonconvex polynomial dependence on cross-sectional areas, certified bounds on global minimizers are obtained via the moment-sum-of-squares hierarchy of semidefinite relaxations. Benchmark studies on 10- and 35-segment frames corroborate the theory. For the 10-segment problem, we obtain the global optimum; for the 35-segment frame, we find high-quality locally-optimal designs that substantially improve on the best known one. We further fabricate and experimentally validate an additional design that closely matches the predictions of the model.
\end{abstract}

\begin{keyword}
Dynamic compliance \sep peak input power \sep frame topology optimization \sep free-vibration eigenvalue \sep robust optimization \sep polynomial optimization
\end{keyword}

\maketitle

\section{Introduction}

Designing minimum-weight frame structures that remain safe and efficient under dynamic excitations is a central challenge in structural optimization \citep{Jog2002,Niu2017,Silva2019}. Beyond classical static constraints, engineers need to control response quantities driven by harmonic loads, such as dynamic compliance (maximum instantaneous force-displacement product) \citep{Olhoff2014} and power (maximum instantaneous force-velocity product) \citep{Heidari2009optimization}. Both quantities govern vibration levels, fatigue, and energy transfer, and thus strongly influence weight-optimal designs.

\subsection{Topology optimization under dynamic constraints}

For truss topology optimization using ground structure discretization, convex formulations for static compliance \citep{BenTal2001}, multi-load robustness \citep{BenTal1997, Kanno2015}, and eigenfrequency constraints \citep{Achtziger2008} have been developed, enabling reliable global optimization via interior-point methods.

For harmonic loads, the use of energy-based and power-based metrics as objective functions in dynamic topology optimization is well-established. A seminal contribution by \citet{Jog2002} proposed optimizing the average input power, a global measure that effectively reduces vibrations by shifting natural frequencies away from the excitation frequency. Another common objective is the minimization of dynamic compliance \citep{Olhoff2014}, which quantifies the maximum displacement response under harmonic loading. However, the dynamic compliance functional has a significant drawback \citep{Silva2019}: for excitation frequencies above the first resonance, optimization converges to antiresonances that are typically associated with poorly defined topologies. \citet{Silva2020} demonstrated that the use of the active input power, the real part of the complex input power, is more robust because its frequency spectrum is free of such antiresonances.

For peak input power under harmonic loads, \citet{Heidari2009optimization} showed that when (i) all nodal forces are in phase and (ii) the driving frequency stays below the first resonance, the peak input power minimization problem admits a clean semidefinite programming (SDP) reformulation; they also derived the worst-case formulation over unit-magnitude forces that remains SDP-representable. Building on this, \citet{Ma2025} recently developed SDP-representable relaxations for general (possibly out-of-phase, multi-frequency) harmonic loads acting below the first resonance using semidefinite-representable functions and nonnegativity of trigonometric polynomials. Their approach reduces to Heidari's in the in-phase/single-frequency case. However, extension to loads above the first resonance remains an open problem.

In contrast to trusses, whose stiffness is linear in the element areas $a_e$, the Euler--Bernoulli frame stiffness is polynomial in the cross-sectional parameters because the bending stiffness involves the second moment of area $I_e$ \citep[Appendix A]{Tyburec2024global}; with selfsimilar scaling $I_e \propto a_e^2$, while one-direction scaling (e.g., changing height only) gives $I_e \propto a_e^3$. Consequently, free‑vibration–constrained weight minimization is nonconvex and may have feasible sets disconnected due to local resonance of slender members \citep{Ni2014}, and repeated eigenvalues \citep{Achtziger2008}. Relaxation strategies and nonlinear SDP methods have been proposed to navigate this landscape \citep{Yamada2015,Tyburec2024global}. Recent advances show that such frame problems can nevertheless be tackled \emph{globally} via polynomial optimization: the problem is cast as a polynomial SDP over a compact semialgebraic set and then solved by a hierarchy of moment-sum-of-squares (mSOS) relaxations that produce monotone lower bounds and, with suitable constructions, matching feasible upper bounds and certificates of global $\varepsilon$-optimality \citep{Tyburec2024global}.

\subsection{Aims and contributions}
Motivated by the above, we develop a unified treatment of dynamic compliance and peak input power constraints for undamped frame structures under single-frequency, uniform-phase subresonant harmonic loads. Our contributions are:

\begin{itemize}
  \item \textbf{Unified semidefinite reformulations.} We derive SDP constraints for (robust) dynamic compliance and (robust) peak input power under ellipsoidal load uncertainty and show that both are \emph{equivalent} to a free-vibration eigenvalue constraint with a modified, design-independent mass term. This puts dynamic objectives in the same class as standard free-vibration eigenvalue constraints.
  \item \textbf{Mode-worst-case equivalence.} We prove that the worst-case displacement/velocity amplitudes associated with the robust dynamic compliance and peak input power constraints coincide with the fundamental eigenmode of the equivalent free-vibration problem, which clarifies the structure of worst-case states and links the three performance measures.
  \item \textbf{Global optimization pipeline.} Because the equivalent constraint has the free-vibration eigenvalue form, we can deploy the global polynomial-optimization framework for frames (mSOS hierarchy on an Archimedean set with a non-mixed-term monomial basis as presented in \citep{Tyburec2024global}) to obtain certified lower/upper bounds and, in small benchmarks, global optimality certificates.
  \item \textbf{Continuity to the static case.} Taking $\omega\to 0$ in the SDP formulations recovers the classical robust static compliance constraint of \citet{BenTal1997}, providing a continuous bridge from dynamic compliance and peak input power to robust static compliance.
\end{itemize}

On trusses, our peak input power constraints reduce to the known convex SDP form of \citet{Heidari2009optimization} in the single-frequency, in-phase case (and connect to \citet{Ma2025} in the broader multi-frequency setting). On frames, the key novelty is to embed these dynamic constraints into a free-vibration eigenvalue surrogate with a design-independent mass augmentation, which unlocks global solution strategies previously developed for frame eigenvalue problems \citep{Tyburec2024global}.

The paper is organized as follows. Section~\ref{sec:prelim} introduces notations and preliminaries. Section~\ref{sec:freevib} reviews frame optimization under free-vibration eigenvalue and static-compliance constraints, and Section~\ref{sec:moment} provides a moment-based formulation. We then derive semidefinite programs for dynamic compliance and peak input power (Sections~\ref{sec:dyncompl} and \ref{sec:peakpower}) and prove their equivalence to free-vibration eigenvalue constraints, together with the state fields relations (Section~\ref{sec:relation}). Section~\ref{sec:examples} reports numerical results on two benchmarks ($10$- and $35$-segment frames): we certify global optimality for the former and obtain locally-optimal designs for the latter that are significantly lighter than the state of the art; and we further present an additional design with a small optimality gap, manufactured additively and validated experimentally via modal testing, showing excellent agreement with the model predictions. We conclude in Section~\ref{sec:conclusions}.

\subsection{Notation and preliminaries}\label{sec:prelim}

We use the following notation. We write vectors in bold lowercase (e.g., $\mathbf{u}$) and matrices in bold uppercase (e.g., $\mathbf{K}$). Components use subscripts: $u_i$ and $K_{ij}$. The transpose and conjugate (Hermitian) transpose are written as $(\cdot)\T$ and $(\cdot)\TC$, respectively. For $n\in\mathbb{N}$, $\mathbb{S}^n$ is the space of real symmetric $n\times n$ matrices; $\mathbf{A}\succeq\mathbf{B}$ abbreviates $\mathbf{A}-\mathbf{B}\in\mathbb{S}^n_{\succeq 0}$ (the cone of positive semidefinite matrices). The range and kernel of $\mathbf{A}$ are $\range{\mathbf{A}}$ and $\kernel{\mathbf{A}}$. The Moore-Penrose pseudoinverse $\mathbf{A}^\dagger$ is the unique matrix defined by:

\begin{definition}[Moore-Penrose pseudoinverse \textup{\citep{Boyd2004}}]\label{def:pseudo-inverse}
Let $\mathbf{A}\in\mathbb{R}^{m\times n}$. The pseudoinverse $\mathbf{A}^\dagger$ is the unique matrix satisfying
$$
\mathbf{A}\mathbf{A}^\dagger\mathbf{A}=\mathbf{A},\quad
\mathbf{A}^\dagger\mathbf{A}\mathbf{A}^\dagger=\mathbf{A}^\dagger,\quad
(\mathbf{A}\mathbf{A}^\dagger)\T=\mathbf{A}\mathbf{A}^\dagger,\quad
(\mathbf{A}^\dagger\mathbf{A})\T=\mathbf{A}^\dagger\mathbf{A}.
$$
\end{definition}
Furthermore, we recall the Schur complement and its generalization:
\begin{lemma}[Schur complement \textup{\citep[Theorem~1.12]{Horn2005}}]\label{lem:schur_complement}
Let $\mathbf{M}=\begin{psmallmatrix}\mathbf{A}&\mathbf{B}\T\\ \mathbf{B}&\mathbf{C}\end{psmallmatrix}\in\mathbb{S}^{m+n}$. If $\mathbf{C}\succ\mathbf{0}$, then
$$
\mathbf{M}\succeq\mathbf{0}\ \Longleftrightarrow\ \mathbf{A}-\mathbf{B}\mathbf{C}^{-1}\mathbf{B}\T\succeq\mathbf{0}.
$$
\end{lemma}

\begin{lemma}[Generalized Schur complement \textup{\citep[Theorem~1.20]{Horn2005}}]\label{lem:gen_schur_complement}
Let $\mathbf{M}=\begin{psmallmatrix}\mathbf{A}&\mathbf{B}\T\\ \mathbf{B}&\mathbf{C}\end{psmallmatrix}\in\mathbb{S}^{m+n}$. If $\mathbf{C}\succeq\mathbf{0}$, then the following are equivalent:
\begin{enumerate}
\item $\mathbf{M}\succeq\mathbf{0}$,
\item $\mathbf{A}-\mathbf{B}\mathbf{C}^\dagger\mathbf{B}\T\succeq\mathbf{0}$ and $\range{\mathbf{B}}\subseteq\range{\mathbf{C}}$.
\end{enumerate}
\end{lemma}

\section{Background: Weight optimization under free-vibration eigenvalue and static compliance constraints}

Weight optimization of frame structures under free-vibration eigenvalue and static compliance constraints represents a fundamental challenge in structural engineering. In this section, we summarize the results for this problem as developed in \citep{Tyburec2024global}.

\subsection{Optimization problem formulation}\label{sec:freevib}

Consider a frame structure discretized into $\nel$ Euler--Bernoulli finite elements. The primary design variables are the cross-sectional areas $\mathbf{a} \in \mathbb{R}^{\nel}_{\geq 0}$ of these elements. The objective is to minimize the structural weight, which is the sum of individual element weights
\begin{equation}
  \sum_{e=1}^{\nel} \rho_e \ell_e a_e,
\end{equation}
where $\rho_e \in \mathbb{R}_{> 0}$ represents the material density and $\ell_e \in \mathbb{R}_{> 0}$ the length of an element $e$. However, the design must also satisfy static and dynamic performance requirements. The static behavior is governed by the equilibrium equation $\mathbf{K}(\mathbf{a})\mathbf{u} = \mathbf{f}$, where $\mathbf{K}(\mathbf{a}) \in \mathbb{S}^{\ndof}_{\succeq 0}$ is the global stiffness matrix, $\mathbf{u} \in \mathbb{R}^{\ndof}$ represents nodal displacements, $\mathbf{f} \in \mathbb{R}^{\ndof}$ denotes external forces, and $\ndof$ the number of degrees of freedom. A crucial distinction from truss structures is that the frame's stiffness matrix includes both membrane and bending effects, leading to a polynomial dependence on the design variables \citep[Appendix A]{Tyburec2024global}
\begin{equation}\label{eq:stiffness}
  \mathbf{K}(\mathbf{a}) = \sum_{e=1}^{\nel} \left[a_e \mathbf{K}^{(1)}_{e} + a_e^2 \mathbf{K}^{(2)}_{e} + a_e^3 \mathbf{K}^{(3)}_{e}\right].
\end{equation}
Here, $a_e \mathbf{K}^{(1)}_{e}$ captures the membrane (axial) stiffness, while $a_e^2 \mathbf{K}^{(2)}_{e}$ and $a_e^3 \mathbf{K}^{(3)}_{e}$ represent bending effects. Hence, the maximal polynomial degree in $\mathbf{K}(\mathbf{a})$ is $\dK=3$. For all these matrices, we have by their construction $\mathbf{K}_e^{(i)} \in \mathbb{S}^{n_\mathrm{dof}}_{\succeq 0}, i=1,2,3$, so they belong to the space of real symmetric positive semidefinite matrices $\mathbb{S}^{\ndof}_{\succeq0}$.

The dynamic behavior is governed by the mass matrix $\mathbf{M}(\mathbf{a}) \in \mathbb{S}^{\ndof}_{\succeq 0}$, which has a linear dependence on the design variables
\begin{equation}\label{eq:mass}
  \mathbf{M}(\mathbf{a}) = \mathbf{M}^{(0)} + \sum_{e=1}^{\nel} a_e \mathbf{M}^{(1)}_{e},
\end{equation}
where $\mathbf{M}^{(0)} \in \mathbb{S}^{\ndof}_{\succeq 0}$ represents nonstructural (e.g., lumped) masses and $\mathbf{M}^{(1)}_{e} \in \mathbb{S}^{\ndof}_{\succeq 0}$ are the (stiffness-consistent) element mass matrices.

With these matrices defined, we can formulate the constraints on structural performance. First, under applied loads $\mathbf{f}$, the compliance (work done by external forces) should not exceed prescribed limit $\overline{c} \in \mathbb{R}_{> 0}$,
\begin{equation}
  \mathbf{f}\T \mathbf{u} \leq \overline{c}.
\end{equation}
Second, the fundamental free-vibration eigenvalue $\lambda_{\min}$ should be at least $\underline{\lambda} \in \mathbb{R}_{> 0}$, which, using the Rayleigh quotient, can be expressed as
\begin{equation}\label{eq:rayleigh}
  \underline{\lambda} \leq \lambda_{\min} := \inf_{\mathbf{w} \in \mathbb{R}^{\ndof} \setminus \kernel{\mathbf{M}(\mathbf{a})}}\frac{\mathbf{w}\T \mathbf{K}(\mathbf{a})\mathbf{w}}{\mathbf{w}\T \mathbf{M}(\mathbf{a})\mathbf{w}}.
\end{equation}
We note that $\lambda_{\min}$ is the smallest nonzero eigenvalue of the generalized eigenvalue problem $\mathbf{K}(\mathbf{a})\mathbf{w}=\lambda\,\mathbf{M}(\mathbf{a})\mathbf{w}$, with an eigenvector $\mathbf{w}\in\mathbb{R}^{\ndof}$. The exclusion of the kernel of $\mathbf{M}(\mathbf{a})$ in \eqref{eq:rayleigh} accounts for possible rigid-body modes, ensuring $\lambda_{\min}$ refers to the smallest positive eigenvalue. Since $\mathbf{K}(\mathbf{a}),\mathbf{M}(\mathbf{a})\in\mathbb{S}^{\ndof}_{\succeq 0}$, the spectrum is real and nonnegative. The value of $\lambda_{\min}$ relates to the first angular and natural frequencies as $\omega_{\min}=\sqrt{\lambda_{\min}}$ and $f_{\min}=\sqrt{\lambda_{\min}}/(2\pi)$. 

Through the generalized Schur complement (Lemma \ref{lem:gen_schur_complement}) and properties of positive semidefinite matrices, this problem can be reformulated as a semidefinite program \citep{Achtziger2008,Tyburec2024global}
\begin{subequations}\label{eq:sdp_fv}
  \begin{align}
    \min_{\mathbf{a} \in \mathbb{R}^{\nel}} & \sum_{e=1}^{\nel} \rho_e \ell_e a_e \\
    \text{s.t.} \quad & \mathbf{K}(\mathbf{a}) - \underline{\lambda} \mathbf{M}(\mathbf{a}) \succeq \mathbf{0},\label{eq:sdp_fv:fv}\\
    &
    \begin{pmatrix} \overline{c} & -\mathbf{f}\T \\ -\mathbf{f} & \mathbf{K}(\mathbf{a})
    \end{pmatrix} \succeq \mathbf{0},\label{eq:sdp_fv:c}\\
    & \mathbf{a} \ge \mathbf{0}.\label{eq:sdp_fv:a}
  \end{align}
\end{subequations}
Constraint \eqref{eq:sdp_fv:fv} is equivalent to enforcing $\mathbf{w}\T \left[\mathbf{K}(\mathbf{a})-\underline{\lambda}\mathbf{M}(\mathbf{a})\right] \mathbf{w}\ge 0$ for all $\mathbf{w}$, i.e., $\underline{\lambda}\le \lambda_{\min}$. Constraint \eqref{eq:sdp_fv:c} is the Schur complement form of $\mathbf{f}\T\mathbf{u}\le \overline{c}$ with equilibrium $\mathbf{K}(\mathbf{a})\mathbf{u}=\mathbf{f}$, and \eqref{eq:sdp_fv:a} restricts the cross-sectional areas to be nonnegative.

Formulation \eqref{eq:sdp_fv} elegantly handles issues of nondifferentiability for repeated eigenvalues, which would be present when relying on \eqref{eq:rayleigh}. However, it remains challenging due to the polynomial dependence of $\mathbf{K}(\mathbf{a})$ on $\mathbf{a}$ and the potential disconnectedness of the feasible set \citep{Ni2014,Tyburec2024global}, necessitating specialized solution techniques.

To make the feasible set compact, which is needed for convergence of relaxations introduced later, we replace \eqref{eq:sdp_fv:a} with compactifying constraints that tighten \eqref{eq:sdp_fv:a} as
\begin{subequations}\label{eq:compact}
  \begin{align}
    a_e \left(\frac{\overline{w}}{\rho_e \ell_e}- a_e\right) & \geq 0, \quad \forall e \in \{1,\dots,\nel\}, \\
    \overline{w} - \sum_{e=1}^{\nel} \rho_e \ell_e a_e & \geq 0,
  \end{align}
\end{subequations}
where $\overline{w} \in \mathbb{R}_{>0}$ represents an upper bound on the structural weight. This bound is initially obtained from solving a convex linearization of problem \eqref{eq:sdp_fv} \citep[Lemma 3.9]{Tyburec2024global} and can be iteratively tightened during the solution process as better feasible designs become available. The first constraint in \eqref{eq:compact} provides bounds on individual cross-sectional areas, while the second one bounds their weighted sum. 
These constraints ensure the feasible set is (algebraically) compact, which is needed for the convergence of the mSOS hierarchy. 


\subsection{Global solution using moment-sum-of-squares hierarchy}\label{sec:moment}

To solve the frame optimization problem globally, we employ polynomial optimization techniques. Here, we closely follow and summarize the results of \citep{Tyburec2024global}. The key idea is to view our problem as optimization over probability measures supported on the compact feasible set, rather than directly over the cross-sectional areas. Thus, instead of finding optimal cross-sectional areas, we work with (pseudo-)moments $\mathbf{y}$ of these measures, where each component represents an expectation of a monomial. This apparently more complex viewpoint leads to a sequence of tractable convex problems that converge to the global optimum.

To convert our polynomial optimization problem into a moment problem, we use the Riesz functional $L_\mathbf{y}$. Let $\mathbb{R}[\mathbf{a}]_r$ denote the space of real polynomials in $\mathbf{a} = (a_1,\dots,a_{n_\mathrm{e}})$ of total degree at most $r$, and let
\begin{equation}
\mathbf{b}_r(\mathbf{a})
  =
  \big( a_1^{\alpha_1}\cdots a_{n_\mathrm{e}}^{\alpha_{n_\mathrm{e}}} \big)_{\lvert\bm{\alpha}\rvert \le r}
\end{equation}
be the associated monomial basis of $\mathbb{R}[\mathbf{a}]_r$, ordered in some fixed but arbitrary way. We introduce a vector
\begin{equation}
  \mathbf{y} = (y_{\bm{\alpha}})_{\lvert\bm{\alpha}\rvert \le r} \in \mathbb{R}^{\dim \mathbb{R}[\mathbf{a}]_r},
\end{equation}
called a pseudo-moment vector, whose components are indexed consistently with the monomials in $\mathbf{b}_r(\mathbf{a})$. The linear functional $L_{\mathbf{y}} : \mathbb{R}[\mathbf{a}]_r \to \mathbb{R}$ is then defined by its action on monomials:
\begin{equation}
  L_{\mathbf{y}}\big(a_1^{\alpha_1}\cdots a_{n_\mathrm{e}}^{\alpha_{n_\mathrm{e}}}\big)
  \;=\;
  y_{\bm{\alpha}},
\end{equation}
i.e., $L_{\mathbf{y}}$ replaces each monomial $a_1^{\alpha_1}\cdots a_{n_\mathrm{e}}^{\alpha_{n_\mathrm{e}}}$ by the component of $\mathbf{y}$ corresponding to the index of that monomial in the basis $\mathbf{b}_r(\mathbf{a})$. By linearity, for any (matrix) polynomial
\begin{equation}
  \mathbf{P}(\mathbf{a}) = \sum_{\lvert\bm{\alpha}\rvert \le r} \mathbf{P}_{\bm{\alpha}}\,
  a_1^{\alpha_1}\cdots a_{n_\mathrm{e}}^{\alpha_{n_\mathrm{e}}
}
\end{equation}
we have
\begin{equation}
  L_{\mathbf{y}}\big(\mathbf{P}(\mathbf{a})\big)
  \;=\;
  \sum_{\lvert\bm{\alpha}\rvert \le r} \mathbf{P}_{\bm{\alpha}}\, y_{\bm{\alpha}}.
\end{equation}

Using this operator and denoting by $\mathbf{b}_r(\mathbf{a})$ the canonical basis, i.e., vector of all monomials up to degree $r$, we obtain a hierarchy of linear semidefinite programming relaxations indexed by the relaxation order $r \ge r_{\min}$
\begin{subequations}\label{eq:sos_fv}
  \begin{align}
    \min_{\mathbf{y}} \quad & L_\mathbf{y}\left(\sum_{e=1}^{\nel} \rho_e \ell_e a_e\right) \\
    \text{s.t.} \quad & L_\mathbf{y}\left(\mathbf{b}_r(\mathbf{a})\mathbf{b}_r(\mathbf{a})\T\right) \succeq \mathbf{0},\label{eq:sos_fv:moment} \\
    & L_\mathbf{y}\left(\mathbf{b}_{r-\lceil \dK/2\rceil}(\mathbf{a})\mathbf{b}_{r-\lceil \dK/2\rceil}(\mathbf{a})\T \otimes \left[\mathbf{K}(\mathbf{a}) - \underline{\lambda} \mathbf{M}(\mathbf{a})\right]\right) \succeq \mathbf{0}, \label{eq:sos_fv:fv}\\
    & L_\mathbf{y}\left(\mathbf{b}_{r-\lceil \dK/2\rceil}(\mathbf{a})\mathbf{b}_{r-\lceil \dK/2\rceil}(\mathbf{a})\T \otimes
      \begin{pmatrix} \overline{c} & -\mathbf{f}\T \\ -\mathbf{f} & \mathbf{K}(\mathbf{a})
    \end{pmatrix}\right) \succeq \mathbf{0}, \label{eq:sos_fv:static}\\
    & L_\mathbf{y}\left(\mathbf{b}_{r-1}(\mathbf{a})\mathbf{b}_{r-1}(\mathbf{a})\T \otimes \left[a_e \left(\frac{\overline{w}}{\rho_e \ell_e} - a_e\right)\right] \right) \succeq \mathbf{0}, \quad \forall e \in \{1,\dots,\nel\}\label{eq:sos_fv:compact} \\
    & L_\mathbf{y}\left(\mathbf{b}_{r-1}(\mathbf{a})\mathbf{b}_{r-1}(\mathbf{a})\T \otimes \left[\overline{w} - \sum_{e=1}^{\nel} \rho_e \ell_e a_e\right] \right) \succeq \mathbf{0},\label{eq:sos_fv:ub} \\
    & y_\mathbf{0} = 1,\label{eq:sos_fv:const}
  \end{align}
\end{subequations}
where $\dK$ is the maximum degree of polynomials in $\mathbf{K}(\mathbf{a})$, $\otimes$ denotes the Kronecker product, and $y_\mathbf{0}$ is the moment corresponding to the constant monomial $1$. The moment matrix \eqref{eq:sos_fv:moment} ensures that the moments correspond to a valid probability measure when $r\rightarrow \infty$ \citep{Henrion2006}. The localizing matrices \eqref{eq:sos_fv:fv}--\eqref{eq:sos_fv:ub} enforce that this measure is supported on the feasible set defined by the polynomial inequalities: matrices \eqref{eq:sos_fv:fv}--\eqref{eq:sos_fv:static} correspond to constraints from \eqref{eq:sdp_fv}, while \eqref{eq:sos_fv:compact}--\eqref{eq:sos_fv:ub} represent the compactification \eqref{eq:compact}. Because $\mathbf{K}(\mathbf{a})$ contains polynomials of degree $\dK$, forming its Kronecker product with degree-$2r$ monomials of $\mathbf{b}_r(\mathbf{a}) \mathbf{b}_r(\mathbf{a})\T$ in \eqref{eq:sos_fv:fv}--\eqref{eq:sos_fv:ub} would exceed degree $2r$. We thus use $\mathbf{b}_{r-\lceil \dK/2\rceil} (\mathbf{a})$ to ensure these constraints involve moments up to order $2r$ only. Finally, constraint \eqref{eq:sos_fv:const} normalizes the probability measure by setting the constant moment to one. When $\dK=2$, the valid relaxation order is thus defined by $r_{\min}= 1$, while $r_{\min}=2$ is needed for $\dK=3$.

Let $\underline{w}^{(r)}$ denote the optimal objective value of the relaxation \eqref{eq:sos_fv}. By construction, $\underline{w}^{(r)}$ is a valid lower bound on the globally optimal structural weight, and the sequence $\{\underline{w}^{(r)}\}_r$ is nondecreasing in $r$ because the feasible set of the $(r{+}1)$-th relaxation contains that of the $r$-th.

Based on the numerical results in \citep{Handa2025,Tyburec2024global}, we adopt the non-mixed-term (NMT) monomial basis
\begin{equation}\label{eq:nmt}
  \mathbf{b}_r(\mathbf{a})^{\mathrm{NMT}}
  =
  \begin{pmatrix}
    1 & a_1 & \cdots & a_{\nel} & a_1^2 & \cdots & a_{\nel}^2 & \cdots & a_{\nel}^r
  \end{pmatrix}\T,
\end{equation}
which excludes mixed monomials such as $a_i a_j$ or $a_i^2 a_j$. This dramatically reduces the sizes of the moment and localizing matrices, improving scalability.

The standard Lasserre theory guarantees asymptotic convergence (as $r\!\to\!\infty$) under Archimedeanity of the feasible set for the canonical monomial basis \citep{Lasserre2001,Henrion2006}. For the NMT basis, a comparable general proof is not available; nonetheless, in our numerical experiments the NMT hierarchy exhibits stable behavior, yields tight lower bounds $\underline{w}^{(r)}$, and numerically converges to feasible solutions for small $r$. We therefore use NMT for efficiency, while validating tightness via the optimality gap described next.

Consider the lower bound $\underline{w}^{(r)}$ from \eqref{eq:sos_fv}. From the first-order moments we extract $\tilde{a}_e=y_{\mathbf{e}_e}$ (with $\mathbf{e}_e$ being the unit vector with $e$-th component equal to one) and compute a feasible upper bound by scaling:
\begin{equation}\label{eq:scaling}
  \delta^*=\argmin_{\delta\ge 1}\{\ \delta:\ \delta\tilde{\mathbf{a}}\ \text{satisfies}\ \eqref{eq:sdp_fv}\ \},
\end{equation}
solved globally by bisection. Based on \citep[Proposition 3.12]{Tyburec2024global}, such $\delta^*$ is guaranteed to exist. Denote $\overline{w}^{(r)}=\sum_{e=1}^{\nel}\rho_e\ell_e\,(\delta^*\tilde{a}_e)$. Then $\overline{w}^{(r)}$ is a feasible upper bound and $\underline{w}^{(r)}$ a valid lower bound on the global optimum. We declare global relative $\varepsilon$-optimality if $\frac{\overline{w}^{(r)}-\underline{w}^{(r)}}{\underline{w}^{(r)}}\le \varepsilon$, which is a sufficient condition. In practice, we further refine $\overline{w}^{(r)}$ by a local optimization (e.g., sequential SDP) initialized at $\delta^*\tilde{\mathbf{a}}$.

This solution approach is summarized in Fig.~\ref{fig:optflow}.

\begin{figure}[!t]
  \centering
  \begin{tikzpicture}[node distance=0.5cm and 1.25cm]
    \tikzstyle{block} = [draw, rectangle, rounded corners=5pt, text width=4.5cm, align=center, minimum height=0.8cm, font=\footnotesize]
    \tikzstyle{decision} = [draw, diamond, text width=1.95cm, align=center, inner sep=0pt, minimum size=2.76cm, aspect=1, font=\footnotesize]
    \tikzstyle{line} = [draw, -latex, thick]

    \node[block] (start) {Find a feasible point for \eqref{eq:sdp_fv} \citep[Lemma 3.9]{Tyburec2024global}};
    \node[block, below=of start] (relaxation) {Solve relaxation \eqref{eq:sos_fv} for lower bound $\underline{w}^{(r)}$};
    \node[block, below=of relaxation] (ub) {Compute upper bound $\overline{w}^{(r)}$ via \eqref{eq:scaling}};
    \node[block, below=of ub] (ubimproved) {Refine upper bound $\overline{w}^{(r)}$ to $\hat{w}^{(r)}$ using local optimization};
    \node[decision, below=of ubimproved] (ubbetter) {$\hat{w}^{(r)} < \overline{w}_{\min}$};
    \node[block, right=of ubbetter] (tighten) {Tighten relaxation bounds \eqref{eq:compact} using $\hat{w}^{(r)}$};
    \node[decision, below=of ubbetter] (check) {$\frac{\hat{w}^{(r)}-\underline{w}^{(r)}}{\underline{w}^{(r)}} \le \varepsilon$};
    \node[block, below=of check] (end) {Return globally $\varepsilon$-optimal solution};
    \node[block, right=41mm of check] (increase) {Increase relaxation order $r$};

    \path[line] (start) -- (relaxation);
    \path[line] (relaxation) -- (ub);
    \path[line] (ub) -- (ubimproved);
    \path[line] (ubimproved) -- (ubbetter);

    \path[line] (ubbetter.east) -- ++(0.8cm,0) node[midway,above]{\footnotesize yes} -- (tighten.west);
    \path[line] (tighten.north) |- node[pos=0.25,right]{} (relaxation.east);

    \path[line] (ubbetter.south) -- ++(0,-0.35cm) node[midway,right]{\footnotesize no} -- (check.north);

    \path[line] (check.south) -- ++(0,-0.25cm) node[midway,right]{\footnotesize yes} -- (end.north);
    \path[line] (check.east) -- ++(1.0cm,0) node[midway,above]{\footnotesize no} -- (increase.west);
    \path[line] (increase.north) |- (relaxation.east);
  \end{tikzpicture}
  \caption{Flowchart of the global solution approach. $\overline{w}_{\min}$ is the tightest upper bound found so far.}
  \label{fig:optflow}
\end{figure}
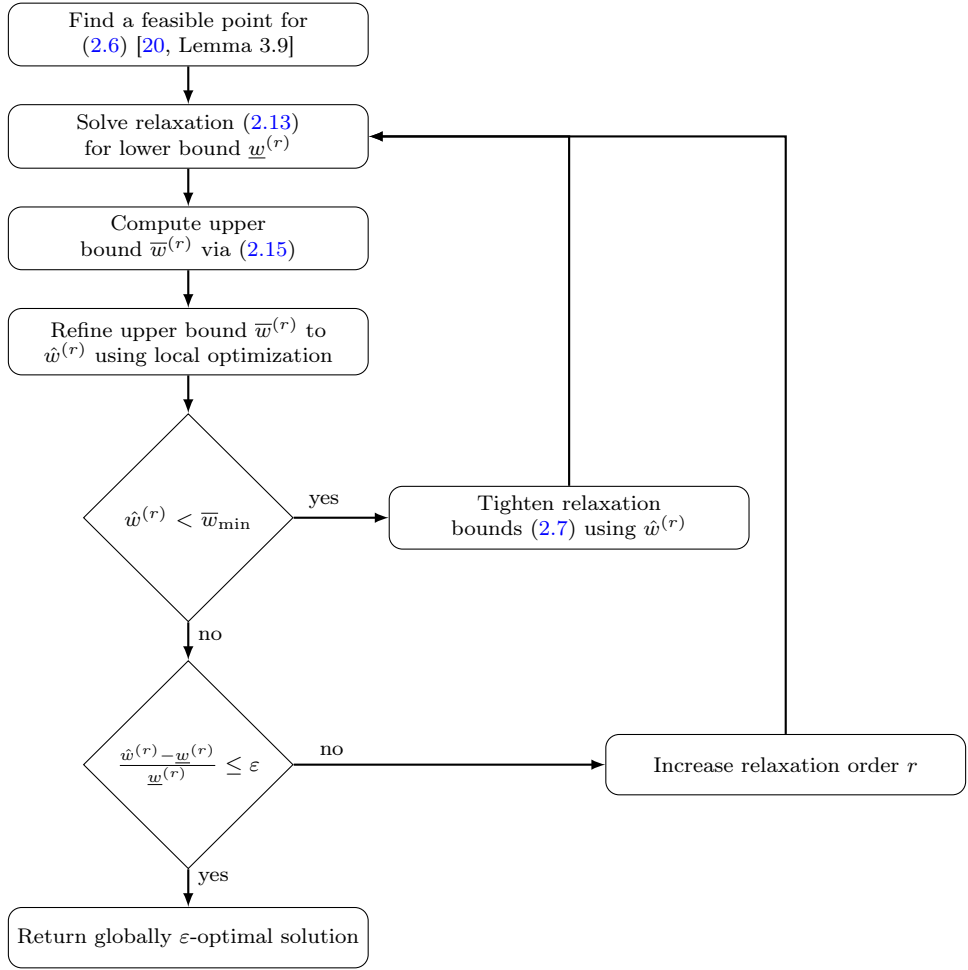

\section{Methods}

In this section, we develop semidefinite programming formulations for the minimization of weight under robust dynamic compliance (Section \ref{sec:dyncompl}) and robust peak input power (Section \ref{sec:peakpower}) constraints under harmonic loads in phase of a frequency below the first resonance. For both of these formulations, we show in Section \ref{sec:relation} that they are equivalent to certain free-vibration eigenvalue constraints of the form \eqref{eq:sdp_fv:fv}, and thus can be handled using the theory developed in \citep{Tyburec2024global} that we summarized in Section \ref{sec:freevib}. Finally, we provide a relation between the worst-case state fields associated with robust dynamic compliance and robust peak input power constraints with the eigenmode of the equivalent free-vibration eigenvalue problem.

\subsection{Robust dynamic compliance constraints under single-frequency, uniform-phase harmonic loads}\label{sec:dyncompl}

We start by developing semidefinite programming reformulation of (robust) dynamic compliance constraints, which are common in topology optimization problems under dynamic loads; see, e.g., \citep{Olhoff2016}. First, we consider the single-load case, where the dynamic compliance is defined as the maximum of the instantaneous force-displacement product \citep{Olhoff2014}. We then extend this formulation to the robust setting, where the dynamic compliance is constrained to be below a given threshold for all possible loads within a given ellipsoidal uncertainty set.

\subsubsection{Single load case semidefinite programming formulation}

We begin with the single load case setting, for which the weight minimization problem under a dynamic compliance constraint reads as
\begin{subequations}\label{eq:dynamiccomplianceopt}
  \begin{align}
    \min_{\mathbf{a}\in \mathbb{R}^{\nel}, \mathbf{u} \in \mathbb{C}^{\ndof}} & \sum_{e=1}^{n} \rho_e \ell_e a_e\\
    \mathrm{s.t.}\; & \mathbf{f}(t)\TC \mathbf{u}(t) \le \overline{d}, \quad\forall t\in \mathbb{R}\label{eq:dynamiccompliance_abs}\\
    & \mathbf{K}(\mathbf{a})\mathbf{u}(t) + \mathbf{M}(\mathbf{a})\mathbf{u}''(t) = \mathbf{f}(t),\label{eq:state_eq}\\
    & \mathbf{a}\ge \mathbf{0},
  \end{align}
\end{subequations}
where $\mathbf{u}(t) \in \mathbb{C}^{\ndof}$ and $\mathbf{f}(t) \in \mathbb{C}^{\ndof}$ are time $t$ dependent harmonic displacements and forces, respectively. In particular, we assume the following form of the transient forces with the angular frequency $\omega \in \mathbb{R}_{>0}$, constants $c_1, c_2 \in \mathbb{R}$ such that $c_1^2+c_2^2=1$, and the amplitude vector $\mathbf{f}_\A \in \mathbb{R}^{\ndof}$:
\begin{equation}\label{eq:forces}
  \begin{aligned}
    \mathbf{f}(t) &= c_1 \cos(\omega t) \mathbf{f}_\A + c_2 \sin(\omega t) \mathbf{f}_\A\\
    &= \frac{1}{2} e^{\omega t i}(c_1 \mathbf{f}_\A - i c_2 \mathbf{f}_\A) + \frac{1}{2} e^{-\omega t i}( c_1 \mathbf{f}_\A + i c_2 \mathbf{f}_\A).
  \end{aligned}
\end{equation}
This formulation expresses the harmonic load as a sum of complex exponentials, which allows us to solve the differential equation \eqref{eq:state_eq} using the method of undetermined coefficients. Since the system is linear, we expect the displacement response to have the same harmonic form as the applied load, but with different amplitudes,
\begin{equation}
  \mathbf{u}(t) = \frac{1}{2} e^{\omega t i}(\mathbf{u}_\Real - i \mathbf{u}_\I) + \frac{1}{2} e^{-\omega t i}( \mathbf{u}_\Real + i \mathbf{u}_\I),
\end{equation}
where $\mathbf{u}_\Real, \mathbf{u}_\I \in \mathbb{R}^{\ndof}$ are the real and imaginary amplitude vectors, respectively. Substituting these expressions for $\mathbf{f}(t)$ and $\mathbf{u}(t)$ into the state equation \eqref{eq:state_eq} yields partitioning based on the real and complex parts as
\begin{subequations}
  \begin{align}
    \left[\mathbf{K}(\mathbf{a}) - \omega^2 \mathbf{M}(\mathbf{a})\right]\mathbf{u}_\Real \left( \frac{1}{2} e^{\omega t i} + \frac{1}{2} e^{-\omega t i}  \right) &= c_1 \mathbf{f}_\A \left( \frac{1}{2} e^{\omega t i} + \frac{1}{2} e^{-\omega t i}  \right),\\
    \left[\mathbf{K}(\mathbf{a}) - \omega^2 \mathbf{M}(\mathbf{a})\right] \mathbf{u}_\I \left( \frac{1}{2} e^{-\omega t i} -\frac{1}{2} e^{\omega t i}  \right)i &=  c_2 \mathbf{f}_\A \left( \frac{1}{2} e^{-\omega t i} - \frac{1}{2} e^{\omega t i} \right)i,
  \end{align}
\end{subequations}
yielding
\begin{subequations}\label{eq:uamp}
  \begin{align}
    \mathbf{u}_\A = \left[ \mathbf{K}(\mathbf{a}) - \omega^2 \mathbf{M}(\mathbf{a})\right]^\dagger \mathbf{f}_\A \quad &\text{if} \quad \mathbf{f}_\A \in \range{\mathbf{K}(\mathbf{a}) - \omega^2 \mathbf{M}(\mathbf{a})},\\
    \mathbf{u}_\A = \emptyset \quad& \text{otherwise},
  \end{align}
\end{subequations}
and
\begin{subequations}
  \begin{align}
    \mathbf{u}_\Real &= c_1 \mathbf{u}_\A,\\
    \mathbf{u}_\I &= c_2 \mathbf{u}_\A.
  \end{align}
\end{subequations}
Here, $(\cdot)^\dagger$ denotes the Moore-Penrose pseudoinverse; under $\omega^2<\lambda_{\min}$ and $\mathbf{a}>\mathbf{0}$ it coincides with the inverse on the relevant subspace. All amplitude vectors are real, and we use $(\cdot)\T$ throughout.

Consequently, the dynamic compliance in \eqref{eq:dynamiccompliance_abs} satisfying \eqref{eq:state_eq} evaluates as
\begin{equation}\label{eq:dynamiccompliance}
  \begin{aligned}
    d &= \sup_{t} \mathbf{f}(t)\TC \mathbf{u}(t)\\
    &
    \begin{multlined}
      = \frac{1}{4} \sup_{t} \left[e^{\omega t i}(c_1 \mathbf{f}_\A - i c_2 \mathbf{f}_\A) + e^{-\omega t i}( c_1 \mathbf{f}_\A + i c_2 \mathbf{f}_\A)\right]\TC\\
      \left[ e^{\omega t i}(c_1 \mathbf{u}_\A - i c_2 \mathbf{u}_\A) + e^{-\omega t i}( c_1 \mathbf{u}_\A + i c_2 \mathbf{u}_\A) \right]
    \end{multlined}\\
    &= \sup_{t} \mathbf{f}_\A\T \left[ \mathbf{K}(\mathbf{a}) - \omega^2 \mathbf{M}(\mathbf{a})\right]^\dagger \mathbf{f}_\A \left[c_1 \cos(\omega t) + c_2 \sin(\omega t)\right]^2.
  \end{aligned}
\end{equation}
To provide a semidefinite programming formulation for the dynamic compliance constraint for $\mathbf{a}\ge \mathbf{0}$, we further restrict $\mathbf{K}(\mathbf{a}) - \omega^2 \mathbf{M}(\mathbf{a}) \succeq 0$. Notice that, based on Section \ref{sec:freevib}, this is equivalent to requiring that $\omega^2 \le \lambda_{\min}$, where $\lambda_{\min}$ is the lowest nonzero eigenvalue of the generalized free-vibration eigenvalue problem $\mathbf{K}(\mathbf{a}) \mathbf{w} - \lambda \mathbf{M}(\mathbf{a})\mathbf{w} = \mathbf{0}$. Such restriction makes the dynamic compliance physically meaningful \citep{Silva2019} and can often be found implicitly implemented, e.g., \citep{Olhoff2016}.

Enforcing this restriction, it follows that $\mathbf{f}_\A\T \left[ \mathbf{K}(\mathbf{a}) - \omega^2 \mathbf{M}(\mathbf{a})\right]^\dagger \mathbf{f}_\A \ge 0$. Because of $c_1^2+c_2^2=1$, we can eliminate explicit dependence on the time $t$ in \eqref{eq:dynamiccompliance} by explicitly evaluating $\sup_t (c_1 \cos(\omega t) + c_2 \sin(\omega t))^2$, yielding
\begin{equation}\label{eq:dcompfinal}
  d = \mathbf{f}_\A\T \left[ \mathbf{K}(\mathbf{a}) - \omega^2 \mathbf{M}(\mathbf{a})\right]^\dagger \mathbf{f}_\A.
\end{equation}
The final expression in equation \eqref{eq:dcompfinal} reveals that the instantaneous force-displacement product reaches its maximum when the phase angle between force and displacement aligns perfectly, which occurs at the time $t = \frac{1}{\omega} \arctan \frac{c_2}{c_1}$ with the period of $\pi/\omega$.

Consequently, we can state the semidefinite version of the dynamic compliance constraints as follows.

\begin{proposition}[Dynamic compliance matrix inequality]\label{prop:dynamic_compliance}
  The following statements are equivalent:
  \begin{itemize}
    \item $\mathbf{f}_\A\T \left[ \mathbf{K}(\mathbf{a}) - \omega^2 \mathbf{M}(\mathbf{a})\right]^\dagger \mathbf{f}_\A \le \overline{d}$, $\mathbf{f}_\A \in \range{\mathbf{K}(\mathbf{a}) - \omega^2 \mathbf{M}(\mathbf{a})}$, $\overline{d}>0$ and $\mathbf{K}(\mathbf{a}) - \omega^2 \mathbf{M}(\mathbf{a})\succeq 0$,
    \item 
      $
      \begin{pmatrix}
        \overline{d} & - \mathbf{f}_\A\T\\
        -\mathbf{f}_\A & \mathbf{K}(\mathbf{a}) - \omega^2 \mathbf{M}(\mathbf{a})
      \end{pmatrix} \succeq 0.$
  \end{itemize}
  \begin{proof}
    The proof follows from a direct application of Lemma \ref{lem:gen_schur_complement}.
  \end{proof}
\end{proposition}

\subsubsection{Robust settings}\label{sec:dyncompl_robust}

The recently developed semidefinite constraint \eqref{eq:dcompfinal} can be extended to a robust setting, where dynamic compliance is constrained to be below a given threshold $\overline{d}_\R$ for all possible loads within a given set of ellipsoidal uncertainties. In the spirit of \citep{BenTal1997}, we restrict the loads within an ellipsoid $\{\mathbf{Q}\mathbf{e}: \mathbf{e} \in \mathbb{R}^{q}, \lVert\mathbf{e}\rVert_2 \le 1\}$ with the matrix $\mathbf{Q} \in \mathbb{R}^{\ndof \times q}$ scaling the $q$-dimensional unit ball $\lVert \mathbf{e}\rVert_2 \le 1$. Hence, we require that
\begin{equation}\label{eq:dyncomplconstr}
  \overline{d}_\R \ge d_\R := \sup_{\lVert \mathbf{e} \rVert_2 \le 1} \sup_{t} \mathbf{f}_\R (t, \mathbf{e})\TC \mathbf{u}(t, \mathbf{e})
\end{equation}
with
\begin{equation}\label{eq:forcerobust}
  \mathbf{f}_\R (t, \mathbf{e}) = c_1 \cos(\omega t) \mathbf{Q} \mathbf{e} + c_2 \sin(\omega t) \mathbf{Q} \mathbf{e}.
\end{equation}
Analogously to the previous section, we obtain
\begin{equation}
  d_\R = \sup_{\lVert\mathbf{e}\rVert_2 \le 1} \left( \mathbf{Q} \mathbf{e}\right)\T \left[ \mathbf{K}(\mathbf{a}) - \omega^2 \mathbf{M}(\mathbf{a}) \right]^\dagger \mathbf{Q} \mathbf{e}
\end{equation}
with $\range{\mathbf{Q}} \subseteq \range{\mathbf{K}(\mathbf{a}) - \omega^2 \mathbf{M}(\mathbf{a})}$. The robust dynamic compliance constraint \eqref{eq:dyncomplconstr} then requires 
\begin{equation}
\{\forall\mathbf{e} \in \mathbb{R}^q: \lVert\mathbf{e}\rVert_2 \le 1\}: \qquad\overline{d}_\R - \mathbf{e}\T\mathbf{Q}\T \left[ \mathbf{K}(\mathbf{a}) - \omega^2 \mathbf{M}(\mathbf{a})\right]^\dagger \mathbf{Q} \mathbf{e} \ge 0.
\end{equation}
After enforcing \(\mathbf{K}(\mathbf{a})-\omega^2\mathbf{M}(\mathbf{a})\succeq0\), we also have $\mathbf{A}:=\mathbf{Q}\T\!\big[\mathbf{K}(\mathbf{a})-\omega^2\mathbf{M}(\mathbf{a})\big]^\dagger\mathbf{Q}\succeq0$. By the Rayleigh-Ritz principle, $d_\R=\sup_{\|\mathbf{e}\|_2\le 1}\mathbf{e}\T\mathbf{A}\mathbf{e}=\lambda_{\max}(\mathbf{A})$ and hence
\begin{equation}\label{eq:robust_dyncompl_sphere}
  \overline{d}_\R \mathbf{I}_q - \mathbf{Q}\T \left[ \mathbf{K}(\mathbf{a}) - \omega^2 \mathbf{M}(\mathbf{a})\right]^\dagger \mathbf{Q} \succeq 0.
\end{equation}
This formulation transforms our constraint from an infinite set of inequalities (one for each possible load in the uncertainty set) into a single matrix inequality constraint.

Applying the generalized Schur complement (Lemma \ref{lem:gen_schur_complement}) to \eqref{eq:robust_dyncompl_sphere} yields
\begin{equation}\label{eq:robustdynamiccompliance}
  \begin{pmatrix}
    \overline{d}_\R\mathbf{I}_q & -\mathbf{Q}\T\\
    -\mathbf{Q}  & \mathbf{K}(\mathbf{a}) - \omega^2 \mathbf{M}(\mathbf{a})
  \end{pmatrix}\succeq 0.
\end{equation}
Finally, we note how to recover the worst-case load amplitudes $\mathbf{f}_\R$ from the set $\{\mathbf{Q}\mathbf{e}: \mathbf{e} \in \mathbb{R}^{q}, \lVert\mathbf{e}\rVert_2 \le 1\}$ and compute the worst-case displacement amplitudes $\mathbf{u}_\R$.

\begin{lemma}[Worst-case load and dynamic compliance]\label{lemma:worstload}
Let \eqref{eq:robustdynamiccompliance} hold and set $\mathbf{A}:=\mathbf{Q}\T\big[\mathbf{K}(\mathbf{a})-\omega^2\mathbf{M}(\mathbf{a})\big]^\dagger\mathbf{Q}\in\mathbb{S}^q_{\succeq 0}$. If $\mathbf{r}_q\in\mathbb{R}^q$ is a unit eigenvector of $\mathbf{A}$ associated with its largest eigenvalue, then the worst-case load is $\mathbf{f}_\R=\mathbf{Q}\mathbf{r}_q$, and the corresponding worst-case dynamic compliance is equal to $\lambda_{\max}(\mathbf{A})$.
\end{lemma}
\begin{proof}
By the Rayleigh-Ritz principle, $\sup_{\lVert\mathbf{e}\rVert_2\le 1}\mathbf{e}\T \mathbf{A}\mathbf{e}=\lambda_{\max}(\mathbf{A})$, attained at any unit eigenvector $\mathbf{r}_q$ of $\mathbf{A}$. The load amplitude is $\mathbf{Q}\mathbf{r}_q$.
\end{proof}

Using the result of this lemma and \eqref{eq:uamp}, the worst-case amplitudes of displacements $\mathbf{u}_\R \in \mathbb{R}^{\ndof}$ evaluate as
\begin{equation}
  \mathbf{u}_\R = \left[ \mathbf{K}(\mathbf{a}) - \omega^2 \mathbf{M}(\mathbf{a})\right]^\dagger \mathbf{Q} \mathbf{r}_{q}.
\end{equation}

\subsection{Robust peak input power constraints under single-frequency, uniform-phase harmonic loads}\label{sec:peakpower}

Next, we provide a semidefinite constraint for the (robust) peak input power that is bounded from above by $\overline{p}$. Since derivation of this constraint is analogous to dynamic compliance presented in the previous section, we provide less details and highlight the differences. The peak input power constraints have also been developed by \citet{Heidari2009optimization}. However, our derivation allows for true topology optimization problems with vanishing cross-sectional areas, i.e., $\mathbf{a}\ge \mathbf{0}$ instead of $\mathbf{a}>\mathbf{0}$, which is crucial when treating oscillating slender beam elements.

Peak input power is particularly important in design scenarios where energy dissipation capacity is limited, such as in lightweight aerospace structures, vibration-sensitive equipment, or structures with fatigue concerns. By constraining peak input power rather than just displacements, designers can directly control the energy transferred into the structure during dynamic loading events.

\subsubsection{Single load case semidefinite programming formulation}

Let us start with the single load case weight minimization problem under a peak input power constraint, which can be formalized as
\begin{subequations}\label{eq:peakpoweropt}
  \begin{align}
    \min_{\mathbf{a}\in \mathbb{R}^{\nel}, \mathbf{v} \in \mathbb{C}^{\ndof}} & \sum_{e=1}^{n} \rho_e \ell_e a_e\\
    \mathrm{s.t.}\; & \lvert \mathbf{f}(t)\TC \mathbf{v}(t) \rvert \le \overline{p}, \quad\forall t\in \mathbb{R}\label{eq:peakpower_abs}\\
    & \mathbf{K}(\mathbf{a})\mathbf{v}(t) + \mathbf{M}(\mathbf{a})\mathbf{v}''(t) = \mathbf{f}'(t),\label{eq:velocity_eq}\\
    & \mathbf{a}\ge \mathbf{0},
  \end{align}
\end{subequations}
where $\mathbf{f}(t) \in \mathbb{C}^{\ndof}$ and $\mathbf{v}(t) \in \mathbb{C}^{\ndof}$ are time $t$ dependent harmonic loads and velocities, respectively. For the harmonic forces as defined in \eqref{eq:forces} and for $\{c_1,c_2 \in \mathbb{R}: c_1^2+c_2^2 = 1\}$, the time derivative $\mathbf{f}'(t)$ reads as
\begin{equation}
  \mathbf{f}'(t) = \frac{1}{2} \omega e^{\omega t i}(i c_1 \mathbf{f}_\A + c_2 \mathbf{f}_\A) + \frac{1}{2}\omega e^{-\omega t i}( - i c_1 \mathbf{f}_\A + c_2 \mathbf{f}_\A).
\end{equation}
Consequently, $\mathbf{v}(t)$ solving \eqref{eq:velocity_eq} takes the form
\begin{equation}
  \mathbf{v}(t) = \frac{1}{2} e^{\omega t i}(\mathbf{v}_\Real - i \mathbf{v}_\I) + \frac{1}{2} e^{-\omega t i}( \mathbf{v}_\Real + i \mathbf{v}_\I),
\end{equation}
with $\mathbf{v}_\Real, \mathbf{v}_\I \in \mathbb{R}^{\ndof}$ being the real and imaginary amplitude vectors, respectively. Substituting these expressions for $\mathbf{f}'(t)$ and $\mathbf{v}(t)$ into the state equation \eqref{eq:velocity_eq} yields
\begin{subequations}
  \begin{align}
    \left[\mathbf{K}(\mathbf{a}) - \omega^2 \mathbf{M}(\mathbf{a})\right]\mathbf{v}_\Real \left( \frac{1}{2} e^{\omega t i} + \frac{1}{2} e^{-\omega t i}  \right) &=\omega c_2 \mathbf{f}_\A \left( \frac{1}{2} e^{\omega t i} + \frac{1}{2} e^{-\omega t i}  \right),\\
    \left[\mathbf{K}(\mathbf{a}) - \omega^2 \mathbf{M}(\mathbf{a})\right] \mathbf{v}_\I \left( -\frac{1}{2} e^{\omega t i} + \frac{1}{2} e^{-\omega t i}  \right)i &= \omega c_1 \mathbf{f}_\A \left( \frac{1}{2} e^{\omega t i} - \frac{1}{2} e^{-\omega t i}  \right)i.
  \end{align}
\end{subequations}
Therefore, the velocity amplitudes $\mathbf{v}_\A$ evaluate as
\begin{subequations}\label{eq:vamp}
  \begin{align}
    \mathbf{v}_\A = \left[ \mathbf{K}(\mathbf{a}) - \omega^2 \mathbf{M}(\mathbf{a})\right]^\dagger \omega \mathbf{f}_\A \quad &\text{if} \quad \mathbf{f}_\A \in \range{\mathbf{K}(\mathbf{a}) - \omega^2 \mathbf{M}(\mathbf{a})},\\
    \mathbf{v}_\A = \left\{\emptyset\right\} \quad& \text{otherwise},
  \end{align}
\end{subequations}

with
\begin{subequations}
  \begin{align}
    \mathbf{v}_\Real &= c_2 \mathbf{v}_\A,\\
    \mathbf{v}_\I &= -c_1 \mathbf{v}_\A.
  \end{align}
\end{subequations}
The velocity amplitude in equation \eqref{eq:vamp} is proportional to both the force amplitude and the excitation frequency $\omega$. Physically, this reflects that higher frequency vibrations produce larger velocities for the same displacement amplitude. The velocity field is phase-shifted relative to the displacement field by 90 degrees, which is characteristic of harmonic motion.

Using these velocity amplitudes, we can now formulate the peak input power, which represents the maximum rate of energy transfer between the external loads and the structure during a vibration cycle. Excluding the nonphysical scenario of $\mathbf{f}_\A \notin \range{\mathbf{K}(\mathbf{a}) - \omega^2 \mathbf{M}(\mathbf{a})}$ while accounting for \eqref{eq:velocity_eq}, the peak input power from \eqref{eq:peakpower_abs} evaluates as
\begin{equation}\label{eq:peakpower}
  \begin{aligned}
    p &= \sup_{t} \left\lvert \mathbf{f}(t)\TC \mathbf{v}(t)\right\rvert\\
    &=
    \begin{multlined}
      \frac{1}{4} \sup_{t} \left\lvert\left[e^{\omega t i}(c_1 \mathbf{f}_\A - i c_2 \mathbf{f}_\A) + e^{-\omega t i}( c_1 \mathbf{f}_\A + i c_2 \mathbf{f}_\A)\right]\TC\right.\\
      \left.\left[ e^{\omega t i}(c_2\mathbf{v}_\A + i c_1 \mathbf{v}_\A) + e^{-\omega t i}( c_2 \mathbf{v}_\A - i c_1 \mathbf{v}_\A) \right]\right\rvert
    \end{multlined}\\
    &= \frac{1}{2}\omega \sup_{t} \left\lvert \mathbf{f}_\A\T \left[ \mathbf{K}(\mathbf{a}) - \omega^2 \mathbf{M}(\mathbf{a})\right]^\dagger \mathbf{f}_\A \left[\left(c_2^2-c_1^2\right)\sin(2 \omega t) + 2 c_1 c_2 \cos(2 \omega t)\right]\right\rvert.
  \end{aligned}
\end{equation}

To provide a semidefinite programming reformulation of \eqref{eq:peakpoweropt} analogous to \citep{Heidari2009optimization} for $\mathbf{a}\ge \mathbf{0}$, we again need to enforce subresonant harmonic loads, i.e., $\mathbf{K}(\mathbf{a}) - \omega^2 \mathbf{M}(\mathbf{a}) \succeq 0$. Then, after recognizing that $\mathbf{f}_\A\T \left[ \mathbf{K}(\mathbf{a}) - \omega^2 \mathbf{M}(\mathbf{a})\right]^\dagger \mathbf{f}_\A \ge 0$ and utilizing the phase-shift formula,
\begin{equation}\label{eq:phaseshift}
 \sup_{t} \left\lvert(c_2^2 - c_1^2) \sin(2\omega t) + 2 c_1 c_2 \cos(2\omega t)\right\rvert = \sup_t \left\lvert\sin\left(2\omega t + \arctan\left(\frac{2 c_1 c_2}{c_2^2 - c_1^2}\right)\right)\right\rvert = 1,
\end{equation}
the peak input power evaluates as
\begin{equation}\label{eq:peakpowerfinal}
  p = \frac{1}{2} \omega \mathbf{f}_\A\T \left[ \mathbf{K}(\mathbf{a}) - \omega^2 \mathbf{M}(\mathbf{a})\right]^\dagger \mathbf{f}_\A.
\end{equation}
Due to the absolute value in \eqref{eq:phaseshift}, this peak value occurs at the time $t = \frac{1}{2\omega} \arctan \frac{2c_1 c_2}{c_2^2 - c_1^2} + \frac{k\pi}{2\omega}, k\in\mathbb{Z}$.

Similarly to the dynamic compliance, we thus receive the following semidefinite programming formulation for the peak input power constraint.

\begin{proposition}[Peak input power matrix inequality]\label{prop:peak_power}
  The following statements are equivalent:
  \begin{itemize}
    \item $\frac{1}{2} \omega \mathbf{f}_\A\T \left[ \mathbf{K}(\mathbf{a}) - \omega^2 \mathbf{M}(\mathbf{a})\right]^\dagger \mathbf{f}_\A \le \overline{p}$, $\mathbf{f}_\A \in \range{\mathbf{K}(\mathbf{a}) - \omega^2 \mathbf{M}(\mathbf{a})}$, $\overline{p}>0$, $\mathbf{K}(\mathbf{a}) - \omega^2 \mathbf{M}(\mathbf{a})\succeq 0$,
    \item 
      $
      \begin{pmatrix}
        \frac{2\overline{p}}{\omega} & - \mathbf{f}_\A\T\\
        -\mathbf{f}_\A & \mathbf{K}(\mathbf{a}) - \omega^2 \mathbf{M}(\mathbf{a})
      \end{pmatrix} \succeq 0.$
  \end{itemize}
  \begin{proof}
    The proof follows again from Lemma \ref{lem:gen_schur_complement}.
  \end{proof}
\end{proposition}

\subsubsection{Robust settings}\label{sec:peakpower_robust}

Let us now proceed with the robust setting of the peak input power constraint, with the loads within an ellipsoid $\{\mathbf{Q}\mathbf{e}: \mathbf{e} \in \mathbb{R}^{q}, \lVert\mathbf{e}\rVert_2 \le 1\}$. Hence, we have
\begin{equation}
  \overline{p}_\R \ge p_\R := \sup_{\lVert \mathbf{e} \rVert_2 \le 1} \sup_{t} \lvert \mathbf{f}_\R (t, \mathbf{e})\TC \mathbf{v}(t, \mathbf{e})\rvert
\end{equation}
with $\mathbf{f}_\R$ defined in \eqref{eq:forcerobust} (the ellipsoidal uncertainty representation from Section \ref{sec:dyncompl_robust}). Using developments analogous to the previous section, we obtain
\begin{equation}
  p_\R = \frac{1}{2}\omega \sup_{\lVert\mathbf{e}\rVert_2 \le 1} \left(\mathbf{Q}\mathbf{e}\right)\T \left[ \mathbf{K}(\mathbf{a}) - \omega^2 \mathbf{M}(\mathbf{a})\right]^\dagger \mathbf{Q} \mathbf{e}
\end{equation}
with $\range{\mathbf{Q}} \subseteq \range{\mathbf{K}(\mathbf{a}) - \omega^2 \mathbf{M}(\mathbf{a})}$. The robust peak input power constraint then requires
\begin{equation}
  \forall \{\mathbf{e} \in \mathbb{R}^q: \lVert\mathbf{e}\rVert_2 \le 1\}: \qquad\overline{p}_\R - \frac{\omega}{2}\mathbf{e}\T\mathbf{Q}\T \left[ \mathbf{K}(\mathbf{a}) - \omega^2 \mathbf{M}(\mathbf{a})\right]^\dagger \mathbf{Q} \mathbf{e} \ge 0,
\end{equation}
which is, after enforcing $\mathbf{K}(\mathbf{a}) - \omega^2 \mathbf{M}(\mathbf{a})\succeq 0$ and similar arguments to Section \ref{sec:dyncompl_robust}, equivalent to a semidefinite constraint
\begin{equation}\label{eq:robust_peakpower_sphere}
  \overline{p}_\R \mathbf{I}_q - \frac{\omega}{2} \mathbf{Q}\T \left[ \mathbf{K}(\mathbf{a}) - \omega^2 \mathbf{M}(\mathbf{a})\right]^\dagger \mathbf{Q} \succeq 0
\end{equation}
and further to
\begin{equation}\label{eq:robust_peakpower}
  \begin{pmatrix}
    \frac{2 \overline{p}_\R}{\omega}\mathbf{I}_q & - \mathbf{Q}\T\\
    -\mathbf{Q} & \mathbf{K}(\mathbf{a}) - \omega^2 \mathbf{M}(\mathbf{a})
  \end{pmatrix} \succeq 0.
\end{equation}
The robust peak input power constraint in equation \eqref{eq:robust_peakpower} ensures that for any loading scenario within the defined uncertainty set, the peak input power remains below the specified threshold $\overline{p}_\R$.

Finally, we can recover the worst-case load amplitudes $\mathbf{f}_\R$ from the set $\{\mathbf{Q}\mathbf{e}: \mathbf{e} \in \mathbb{R}^{q}, \lVert\mathbf{e}\rVert_2 \le 1\}$ and compute the worst-case velocity amplitudes $\mathbf{v}_\R$. The worst-case load is again evaluated as
\begin{equation}
  \mathbf{f}_\R = \mathbf{Q} \mathbf{r}_q,
\end{equation}
which can be derived analogously to Lemma \ref{lemma:worstload}, and the corresponding worst-case velocity amplitudes $\mathbf{v}_\R \in \mathbb{R}^{\ndof}$ based on \eqref{eq:vamp} as
\begin{equation}
  \mathbf{v}_\R = \omega \left[ \mathbf{K}(\mathbf{a}) - \omega^2 \mathbf{M}(\mathbf{a})\right]^\dagger \mathbf{Q} \mathbf{r}_q.
\end{equation}

\subsection{Relations between dynamic compliance, peak input power, free-vibration eigenvalue, and static compliance constraints}\label{sec:relation}

In what follows, we establish relationships between robust dynamic compliance, robust peak input power constraints, fundamental free-vibration eigenvalue constraints, and static compliance constraints. We demonstrate that dynamic compliance and peak input power constraints are equivalent for specific parameter choices, and that both can be expressed in the form of fundamental free-vibration eigenvalue constraints. Furthermore, we prove that the worst-case state fields (displacements and velocities) correspond precisely to the eigenmode associated with the lowest positive eigenvalue of the equivalent free-vibration eigenvalue problem. Finally, we also demonstrate that robust dynamic compliance constraints naturally reduce to robust static compliance constraints as the excitation frequency approaches zero, completing the relationship chain between dynamic, free-vibration, and static structural responses. For reader's convenience, we illustrate the constraint relationships in Fig.~\ref{fig:constraint_relationships}.

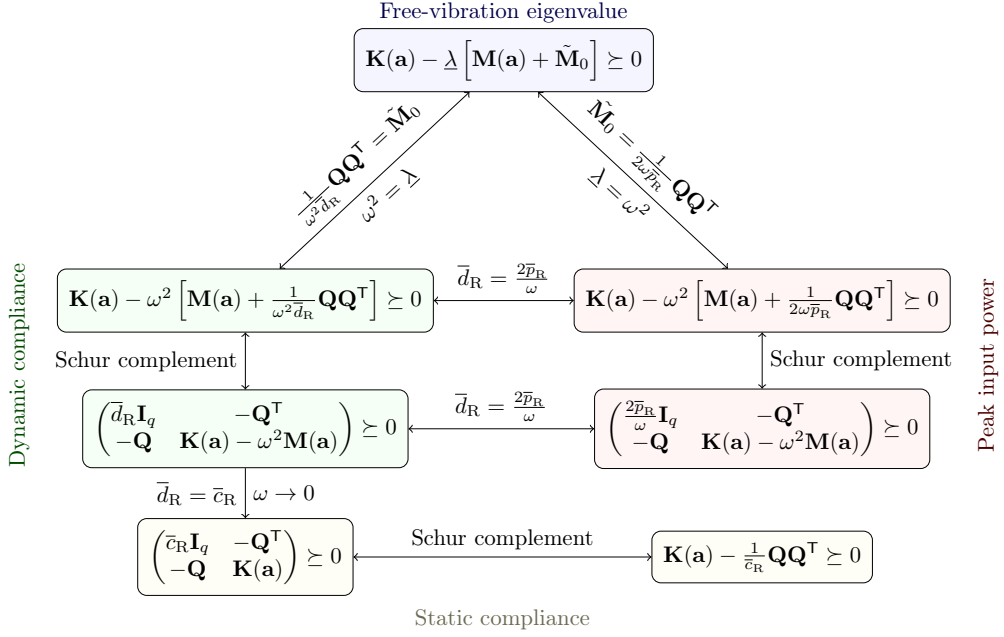
\begin{figure}[!htbp]
  \centering
  \scalebox{0.85}{
    \begin{tikzpicture}
      \node[draw, fill=blue!4, rectangle, rounded corners, inner sep=5pt] (fv) at (0,0.75) {$\mathbf{K}(\mathbf{a}) - \underline{\lambda}\left[\mathbf{M}(\mathbf{a}) + \tilde{\mathbf{M}}_0\right] \succeq 0$};
      \node[draw, fill=green!4, rectangle, rounded corners, inner sep=5pt] (dyncomp) at (-4,-3) {$\mathbf{K}(\mathbf{a}) - \omega^2\left[\mathbf{M}(\mathbf{a}) + \frac{1}{\omega^2\overline{d}_\R} \mathbf{Q} \mathbf{Q}\T\right] \succeq 0$};
      \node[draw, fill=green!4, rectangle, rounded corners, inner sep=5pt] (dyncomp2) at (-4,-5) {$
        \begin{pmatrix}\overline{d}_\R \mathbf{I}_q & - \mathbf{Q}\T\\ -\mathbf{Q} & \mathbf{K}(\mathbf{a}) - \omega^2 \mathbf{M}(\mathbf{a})
      \end{pmatrix} \succeq 0$};
      \node[draw, fill=yellow!4, rectangle, rounded corners, inner sep=5pt] (statcomp) at (-4,-7) {$
        \begin{pmatrix}\overline{c}_\R \mathbf{I}_q & - \mathbf{Q}\T\\ -\mathbf{Q} & \mathbf{K}(\mathbf{a})
      \end{pmatrix} \succeq 0$};
      \node[draw, fill=yellow!4, rectangle, rounded corners, inner sep=5pt] (statcomp2) at (4,-7) {$\mathbf{K}(\mathbf{a}) - \frac{1}{\overline{c}_\R}\mathbf{Q} \mathbf{Q}\T \succeq 0$};
      \node[draw, fill=red!4, rectangle, rounded corners, inner sep=5pt] (peakpower) at (4,-3) {$\mathbf{K}(\mathbf{a}) - \omega^2\left[\mathbf{M}(\mathbf{a}) + \frac{1}{2 \omega \overline{p}_\R} \mathbf{Q} \mathbf{Q}\T\right] \succeq 0$};
      \node[draw, fill=red!4, rectangle, rounded corners, inner sep=5pt] (peakpower2) at (4,-5) {$
        \begin{pmatrix}\frac{2 \overline{p}_\R}{\omega}\mathbf{I}_q & - \mathbf{Q}\T\\ -\mathbf{Q} & \mathbf{K}(\mathbf{a}) - \omega^2 \mathbf{M}(\mathbf{a})
      \end{pmatrix} \succeq 0$};
      \draw[<->] (fv) -- (dyncomp) node[midway,sloped,above] {$\frac{1}{\omega^2 \overline{d}_\R} \mathbf{Q} \mathbf{Q}\T = \tilde{\mathbf{M}}_0$} node[midway,sloped,below] {$\omega^2 = \underline{\lambda}$};
      \draw[<->] (fv) -- (peakpower) node[midway,sloped,above] {$\tilde{\mathbf{M}}_0 = \frac{1}{2\omega \overline{p}_\R} \mathbf{Q} \mathbf{Q}\T$} node[midway,sloped,below] {$\underline{\lambda} = \omega^2$};
      \draw[<->] (dyncomp) -- (peakpower) node[midway,above] {$\overline{d}_\R = \frac{2 \overline{p}_\R}{\omega}$};
      \draw[<->] (dyncomp) -- (dyncomp2) node[midway,left] {Schur complement};
      \draw[<->] (peakpower) -- (peakpower2) node[midway,right] {Schur complement};
      \draw[<->] (dyncomp2) -- (peakpower2) node[midway,above] {$\overline{d}_\R = \frac{2 \overline{p}_\R}{\omega}$};
      \draw[->] (dyncomp2) -- (statcomp) node[midway,left] {$\overline{d}_\R = \overline{c}_\R$} node[midway,right] {$\omega\rightarrow0$};
      \draw[<->] (statcomp) -- (statcomp2) node[midway,above] {Schur complement} node[midway,below=7mm] {{\textcolor{yellow!30!black}{Static compliance}}};
      \node[] at (-7.5,-4) {\rotatebox{90}{\textcolor{green!30!black}{Dynamic compliance}}};
      \node[] at (7.5,-4) {\rotatebox{90}{\textcolor{red!30!black}{Peak input power}}};
      \node[] at (0,1.5) {{\textcolor{blue!30!black}{Free-vibration eigenvalue}}};
  \end{tikzpicture}}
  \caption{Schematic illustration of the relationships between different constraint types.}
  \label{fig:constraint_relationships}
\end{figure}

\subsubsection{Relation between (robust) dynamic compliance and (robust) peak input power constraints}

We start by relating the robust dynamic compliance constraints \eqref{eq:robustdynamiccompliance} with the robust peak input power constraints \eqref{eq:robust_peakpower}. Clearly, both these constraints maintain the same structure and are equivalent for $\overline{d}_\R = \frac{2\overline{p}_\R}{\omega}$ if $\omega\in \mathbb{R}_{>0}$. As a direct consequence, we get the following simple result that also relates the state fields:
\begin{lemma}[Displacement-velocity relation]\label{lemma:displacement_velocity}
  Fix $\mathbf a\in\mathbb R_{\ge0}^{\nel}$ and $\omega>0$, and set $\mathbf S:=\mathbf K(\mathbf a)-\omega^2\mathbf M(\mathbf a)\in\mathbb S^{\ndof}$. For any $\mathbf{f}_\R \in\mathbb R^{\ndof}$ define the minimal-norm amplitudes $\mathbf{u}_\R := \mathbf{S}^\dagger \mathbf{f}_\R$ and $\mathbf{v}_\R:=\mathbf{S}^\dagger(\omega \mathbf{f}_\R)$. Then $\mathbf{v}_\R=\omega \mathbf{u}_\R$ and, with $d_\R:=\mathbf{f}_\R\T \mathbf{u}_\R$ and $p_\R:=\frac{1}{2} \mathbf{f}_\R\T \mathbf{v}_\R$, one has $d_\R=\frac{2p_\R}{\omega}$.
\end{lemma}
\begin{proof}
  By linearity of the map $\mathbf{x} \mapsto \mathbf{S}^\dagger \mathbf{x}$, $\mathbf{v}_\R=\omega \mathbf{S}^\dagger \mathbf{f}_\R=\omega \mathbf{u}_\R$. Hence $p_\R=\frac{1}{2} \mathbf{f}_\R\T(\omega \mathbf{u}_\R) = \frac{\omega}{2} d_\R$.
\end{proof}
Despite these equivalences at the robust-constraint level, dynamic compliance and input power are distinct time-domain functionals---cf. \eqref{eq:dynamiccompliance} vs. \eqref{eq:peakpower}---so their instantaneous histories and peaks occur at different phases.

\subsubsection{Relation between (robust) dynamic compliance, (robust) peak input power constraints, and fundamental free-vibration eigenvalue constraints}

In the previous section, we have shown that the robust dynamic compliance and robust peak input power constraints are equivalent for $\omega \in \mathbb{R}_{>0}$ and a specific parameter choice. To relate these constraints to the fundamental free-vibration eigenvalue constraints \eqref{eq:sdp_fv:fv} under $\omega\in \mathbb{R}_{>0}$, we only need to consider one of them. For this, we choose the robust dynamic compliance constraint \eqref{eq:robustdynamiccompliance} and show that it can be written in the form of \eqref{eq:sdp_fv:fv}, and vice versa.

\begin{proposition}
  Constraints \eqref{eq:robustdynamiccompliance} and \eqref{eq:sdp_fv:fv} are equivalent up to a parameter reparametrization.
\end{proposition}
\begin{proof}
  $\Rightarrow$ Let \eqref{eq:robustdynamiccompliance} hold for $\overline{d}_\R \in \mathbb{R}_{>0}$. Then, after using Schur complement (Lemma \ref{lem:schur_complement}) at the top-left block $\overline d_\R \mathbf{I}_q\succ0$ yields
  \begin{equation}
    \mathbf{K}(\mathbf{a}) - \underline{\lambda} \left[\tilde{\mathbf{M}}_0 +  \sum_{e=1}^{\nel} a_e \mathbf{M}_e ^{(1)} \right]  \succeq 0
  \end{equation}
  with $\underline{\lambda} = \omega^2$ and $\tilde{\mathbf{M}}_0 = \mathbf{M}_0 + \frac{1}{\omega^2 \overline{d}_\R} \mathbf{Q} \mathbf{Q}\T$. This is exactly the form of \eqref{eq:sdp_fv:fv}.

  $\Leftarrow$ Let \eqref{eq:sdp_fv:fv} hold with $\underline{\lambda}>0$ and factor $\mathbf{M}_0=\hat{\mathbf{Q}}\hat{\mathbf{Q}}\T$. Then $\mathbf K(\mathbf a)-\underline\lambda\sum_{e=1}^{\nel} a_e\mathbf M^{(1)}_e-\underline\lambda\hat{\mathbf{Q}}\hat{\mathbf{Q}}\T\succeq0$,
  which is equivalent (reverse Schur complement w.r.t.\ $(1/\underline\lambda)\mathbf{I}_q$) to
  \begin{equation}
    \begin{pmatrix}
      \frac{1}{\underline{\lambda}}\mathbf{I}_q & -\hat{\mathbf{Q}}\T\\
      -\hat{\mathbf{Q}} & \mathbf{K}(\mathbf{a})-\underline{\lambda}\sum_{e=1}^{\nel} a_e \mathbf{M}^{(1)}_e
    \end{pmatrix} \succeq 0,
  \end{equation}
  i.e., \eqref{eq:robustdynamiccompliance} with $\omega=\sqrt{\underline\lambda}$, $\mathbf{Q} = \hat{\mathbf{Q}}$, and $\overline d_\R=1/\underline\lambda$.  
\end{proof}

Having related the robust dynamic compliance and fundamental free-vibration eigenvalue constraints, we further prove that the displacement field of the robust dynamic compliance problem related to the worst-case response is also the eigenvector associated with the lowest nonzero eigenvalue of the free-vibration eigenvalue problem.

\begin{lemma}[Displacement-eigenmode relation]\label{lemma:dyncomp_eigenvec}
Fix $\mathbf{a}\in\mathbb R_{\ge0}^{\nel}$ and $\omega>0$, and set $\underline{\lambda}:=\omega^2$ and
$\mathbf{S}:=\mathbf{K}(\mathbf{a})-\underline{\lambda}\mathbf{M}(\mathbf{a})$.
Assume the robust dynamic-compliance inequality \eqref{eq:robustdynamiccompliance} holds with bound $d_\R>0$,
and let the worst-case load be $\mathbf{f}_\R= \mathbf{Q} \mathbf{r}_{q}$ and
$\mathbf{u}_\R:=\mathbf{S}^\dagger\mathbf{f}_\R$.
Then $\mathbf{u}_\R\neq \mathbf{0}$ satisfies
\[
\mathbf{K}(\mathbf{a}) \mathbf{u}_\R
=\underline{\lambda}\Big(\mathbf{M}(\mathbf{a})+\frac{1}{\underline{\lambda} d_\R} \mathbf{Q}\mathbf{Q}\T\Big)\mathbf{u}_\R,
\]
i.e., $\mathbf{u}_\R$ is an eigenvector and it is associated with the smallest positive eigenvalue $\mu=\underline{\lambda}$.
\end{lemma}

\begin{proof}
  Consider the eigendecomposition $\mathbf{Q}\T\left[\mathbf{K}(\mathbf{a}) - \omega^2 \mathbf{M}(\mathbf{a})\right]^\dagger\mathbf{Q} = \mathbf{R} \mathbf{\Lambda} \mathbf{R}\T$, so that
  \begin{equation}
    \mathbf{\Lambda} = \mathbf{R}\T \mathbf{Q}\T \left[\mathbf{K}(\mathbf{a}) - \omega^2 \mathbf{M}(\mathbf{a})\right]^\dagger \mathbf{Q}\mathbf{R}
  \end{equation}
  with the diagonal matrix of eigenvalues $\mathbf{\Lambda}$. Clearly, the vectors $\mathbf{f}_1 = \mathbf{Q} \mathbf{r}_1, \dots, \mathbf{f}_\R = \mathbf{f}_q = \mathbf{Q} \mathbf{r}_q$ (which we further collect in a matrix $\mathbf{F}$) are $\left[\mathbf{K}(\mathbf{a}) - \omega^2 \mathbf{M}(\mathbf{a})\right]^\dagger$-orthogonal, where $\mathbf{r}_q$ is the orthonormal eigenvector related to the maximal eigenvalue in $\mathbf{\Lambda}$. In addition, we also have $\mathbf{F}\mathbf{F}\T = \mathbf{Q} \mathbf{R} \mathbf{R}\T \mathbf{Q}\T = \mathbf{Q} \mathbf{Q}\T$ due to the orthonormality of the columns of $\mathbf{R}$.

  To prove the statement, we first show that $\mathbf{u}_\R$ is an eigenvector, i.e.,
  \begin{equation}
    \mathbf{K}(\mathbf{a}){\mathbf{u}_\R} - \mu \left[ \frac{1}{\underline{\lambda} {d}_\R} \mathbf{Q} \mathbf{Q}\T + \mathbf{M}(\mathbf{a})\right] {\mathbf{u}_\R} = \mathbf{0}
  \end{equation}
  for an eigenvalue $\mu \in \mathbb{R}_{>0}$.
  To this goal, we rewrite the left hand side as
  \begin{equation}
    \left(\mathbf{K}(\mathbf{a}) - \mu\mathbf{M}(\mathbf{a})\right) {\mathbf{u}}_\R - \frac{\mu}{\underline{\lambda} d_\R} \mathbf{F} \mathbf{F}\T\mathbf{u}_\R
  \end{equation}
  and as
  \begin{equation}
    \left(\mathbf{K}(\mathbf{a}) - \mu\mathbf{M}(\mathbf{a})\right) \left( \mathbf{K}(\mathbf{a}) - \underline{\lambda} \mathbf{M}(\mathbf{a}) \right)^\dagger \mathbf{f}_\R - \frac{\mu}{\underline{\lambda} d_\R} \mathbf{F}\mathbf{F}\T\left( \mathbf{K}(\mathbf{a}) - \underline{\lambda} \mathbf{M}(\mathbf{a}) \right)^\dagger \mathbf{f}_\R.
  \end{equation}
  Recalling that $\mathbf{f}_\R = \mathbf{f}_q$, that  $\mathbf{f}_1 \dots \mathbf{f}_q$ are $\left[\mathbf{K}(\mathbf{a}) - \omega^2 \mathbf{M}(\mathbf{a})\right]^\dagger$-orthogonal and $d_\R = \mathbf{f}_\R\T \left[\mathbf{K}(\mathbf{a}) - \omega^2 \mathbf{M}(\mathbf{a})\right]^\dagger \mathbf{f}_\R$, we obtain
  \begin{equation}
    \left(\mathbf{K}(\mathbf{a}) - \mu\mathbf{M}(\mathbf{a})\right) \left( \mathbf{K}(\mathbf{a}) - \underline{\lambda} \mathbf{M}(\mathbf{a}) \right)^\dagger \mathbf{f}_\R - \frac{\mu}{\underline{\lambda}} \mathbf{f}_\R
  \end{equation}
  which evaluates as $\mathbf{0}$ if $\mu = \underline{\lambda}$. This shows that $\underline{\lambda}$ is an eigenvalue associated with the eigenvector $\mathbf{u}_\R$.

  Second, we need to show that $\mu= \underline{\lambda}$ is particularly the smallest eigenvalue. For this, we observe that the smallest positive eigenvalue is the best lower-bound $\mu_{\min}$ to all eigenvalues \citep{BenTal2001}, i.e.,
  \begin{equation}\label{eq:smallesteig}
    \mu_{\min} = \max \left\{\mu \in \mathbb{R}:  \mathbf{K}(\mathbf{a}) - \mu \left[\frac{1}{\underline{\lambda} d_\R} \mathbf{Q} \mathbf{Q}\T + \mathbf{M}(\mathbf{a}) \right]  \succeq0\right\}.
  \end{equation}
  Notice that the constraint in \eqref{eq:smallesteig} is a linear matrix inequality because it is considered for a fixed $\mathbf{a}$. Using Schur complement lemma and recalling that $\mathbf{Q} \mathbf{Q}\T = \mathbf{F} \mathbf{F}\T$, the matrix inequality is equivalent to
  \begin{equation}
    \begin{pmatrix}
      \frac{\underline{\lambda} d_\R}{\mu} \mathbf{I}_q& \mathbf{F}\T\\
      \mathbf{F} & \mathbf{K}(\mathbf{a}) - \mu \mathbf{M}(\mathbf{a})
    \end{pmatrix} \succeq 0
    \quad \Leftrightarrow \quad
    \frac{\underline{\lambda} d_\R}{\mu} \mathbf{I}_q - \mathbf{F}\T \left[\mathbf{K}(\mathbf{a}) - \mu \mathbf{M}(\mathbf{a})\right]^\dagger \mathbf{F} \succeq 0.
  \end{equation}
  Now test $\mu = \underline{\lambda}$, which gives
  \begin{equation}
    d_\R \mathbf{I}_q - \mathbf{F}\T \left[\mathbf{K}(\mathbf{a}) - \underline{\lambda} \mathbf{M}(\mathbf{a})\right]^\dagger \mathbf{F} \succeq 0;
  \end{equation}
  this is positive semidefinite due to equivalence to \eqref{eq:robustdynamiccompliance}. Therefore, $\underline{\lambda}$ is feasible for \eqref{eq:smallesteig} so that $\mu_{\min} \ge \underline{\lambda}$. Since we have already shown that $\underline{\lambda}$ is an eigenvalue, $\mu_{\min} \le \underline{\lambda}$. Combining both results, we conclude that $\mu_{\min} = \underline{\lambda}$, and $\mathbf{u}_\R$ is its eigenvector.
\end{proof}

An analogous result holds for the peak input power constraints as well.

\begin{lemma}[Velocity-eigenmode relation]\label{lemma:peakpower_eigenvec}
Fix $\mathbf{a}\in\mathbb R_{\ge0}^{\nel}$ and $\omega>0$, and set $\underline{\lambda}:=\omega^2$ and
$\mathbf{S}:=\mathbf{K}(\mathbf{a})-\underline{\lambda}\mathbf{M}(\mathbf{a})$.
Assume the robust peak-input-power inequality \eqref{eq:robust_peakpower} holds with bound $p_\R>0$,
and let the worst-case load be $\mathbf{f}_\R= \mathbf{Q} \mathbf{r}_{q}$ and
$\mathbf{v}_\R:=\mathbf{S}^\dagger(\omega\mathbf{f}_\R)$.
Then $\mathbf{v}_\R\neq \mathbf{0}$ satisfies
\[
\mathbf K(\mathbf{a})\mathbf{v}_\R
=\underline{\lambda}\Big(\mathbf{M}(\mathbf{a})+\frac{1}{2\omega p_\R}\,\mathbf{Q}\mathbf{Q}\T\Big)\mathbf{v}_\R,
\]
i.e., $\mathbf{v}_\R$ is an eigenvector and it is associated with the smallest positive eigenvalue $\mu=\underline{\lambda}$.
\end{lemma}
\begin{proof}
  The proof follows from the fact that $\mathbf{u}_\R = \frac{1}{\omega} \mathbf{v}_\R$ due to Lemma \ref{lemma:displacement_velocity} and that $\mathbf{u}_\R$ is the eigenvector as shown in Lemma \ref{lemma:dyncomp_eigenvec}.
\end{proof}


\subsubsection{Relation between (robust) dynamic compliance and (robust) static compliance constraints}

Finally, we also relate the robust dynamic compliance constraint \eqref{eq:robustdynamiccompliance} to the static one \citep{Tyburec2023global}. Specifically, since all blocks depend continuously on $\omega$, the limit $\omega\rightarrow0$ of the block-PSD condition exists and yields the robust static compliance constraint of \citet{BenTal1997}
\begin{equation}\label{eq:robustcompliance}
  \begin{pmatrix}
    \overline{c}_\R\mathbf{I}_q & - \mathbf{Q}\T\\
    -\mathbf{Q} & \mathbf{K}(\mathbf{a})
  \end{pmatrix} \succeq 0
\end{equation}
which enforces the worst-case static compliance to be bounded from above by $\overline{c}_\R \in \mathbb{R}_{>0}$.  After using Lemma \ref{lem:schur_complement}, \eqref{eq:robustcompliance} can be written in the form
\begin{equation}\label{eq:robustcompliance_fv}
  \mathbf{K}(\mathbf{a}) - \mathbf{M}_0 \succeq 0
\end{equation}
with $\mathbf{M}_0 = \frac{1}{\overline{c}_\R}\mathbf{Q} \mathbf{Q}\T$. As already observed by \citet{Kanno2015}, this is a particular instance of the free-vibration eigenvalue constraint \eqref{eq:sdp_fv:fv}. This reveals that the results summarized in Section \ref{sec:freevib} also apply to the static compliance case. However, we note here that a slightly more general result for the robust static case, allowing to include design-independent stiffness, follows from an analysis of the static compliance function \citep{Tyburec2023global}.


\section{Numerical examples}\label{sec:examples}

In this section, we illustrate the theoretical results through numerical examples. The first problem is a $10$-segment frame structure that was previously studied in \citep{Ni2014}. We investigate three equivalent formulations of this problem: free-vibration eigenvalue, robust dynamic compliance, and robust peak input power constraints, and provide a certified global minimizer. The second example is a $35$-segment tower structure also introduced in \citep{Ni2014}. Here, we solve only the free-vibration eigenvalue formulation and interpret the results in terms of the peak input power and dynamic compliance constraints. Final example involves a $12$-segment structure with rectangular cross-sections; we provide a near-optimal solution and validate the design performance through experimental validation.

All calculations were performed on a Dell Latitude 5450 personal laptop equipped with $64$~GB RAM and the Intel$^{\text{\textregistered}}$ Core\texttrademark~Ultra~7 165H processor. MATLAB implementation and input files are available in \href{https://gitlab.com/open-mechanics/software/pof-dyna}{https://gitlab.com/open-mechanics/software/pof-dyna}. Semidefinite programs were solved using the \textsc{Mosek} optimizer with default parameters. We used the non-mixed-term (NMT) monomial basis \eqref{eq:nmt} throughout, with Archimedean compactification \eqref{eq:compact}. Upper bounds were refined by inhouse sequential semidefinite programming solver initialized at the scaled first-order-moment design \eqref{eq:scaling}.

\subsection{10-segment frame}

As the first illustration, we consider a $10$-segment problem introduced in \citep[Section 5.1]{Ni2014}. This benchmark relies on the Koloušek method for dynamic analysis, which yields exact natural frequencies of the continuous frame model considered therein. In contrast, our study adopts a fixed discretization with Euler--Bernoulli finite elements; the resulting eigenfrequencies are thus approximate with respect to the continuous model. Consequently, our global optimality statements apply only to the discretized problem we solve.

\begin{figure}[!b]
\centering
\begin{tikzpicture}
\scaling{2.5};
\point{n1}{0.000000}{0.000000}
\point{n2}{1.000000}{0.000000}
\point{n3}{2.000000}{0.000000}
\point{n4}{0.000000}{1.000000}
\point{n5}{1.000000}{1.000000}
\point{n6}{2.000000}{1.000000}
\point{n7}{0.500000}{0.000000}
\point{n8}{1.500000}{0.000000}
\point{n9}{0.450000}{0.450000}
\point{n9b}{0.55}{0.55}
\point{n10}{0.500000}{0.500000}
\point{n11}{1.000000}{0.500000}
\point{n12}{1.450000}{0.450000}
\point{n12b}{1.550000}{0.550000}
\point{n13}{1.500000}{0.500000}
\point{n14}{2.000000}{0.500000}
\point{n15}{0.500000}{1.000000}
\point{n16}{1.500000}{1.000000}
\beam{2}{n4}{n15}
\beam{2}{n5}{n15}
\notation{4}{n4}{n5}[$1$]
\beam{2}{n4}{n10}
\notation{4}{n4}{n10}[$2$]
\beam{2}{n2}{n10}
\beam{2}{n1}{n9}
\notation{4}{n1}{n9}[$3$]
\beam{2}{n5}{n9b}
\beam{2}{n1}{n7}
\notation{4}{n1}{n2}[$4$][0.55]
\beam{2}{n2}{n7}
\beam{2}{n2}{n11}
\notation{4}{n2}{n5}[$5$][0.55]
\beam{2}{n5}{n11}
\beam{2}{n5}{n16}
\notation{4}{n5}{n6}[$6$]
\beam{2}{n6}{n16}
\beam{2}{n5}{n13}
\notation{4}{n5}{n13}[$7$]
\beam{2}{n3}{n13}
\beam{2}{n2}{n12}
\notation{4}{n2}{n12}[$8$]
\beam{2}{n6}{n12b}
\beam{2}{n2}{n8}
\notation{4}{n2}{n3}[$9$]
\beam{2}{n3}{n8}
\beam{2}{n3}{n14}
\notation{4}{n3}{n6}[$10$]
\beam{2}{n6}{n14}
\draw[ultra thick] (n9b) arc[start angle=45, end angle=225, radius=0.175];
\draw[ultra thick] (n12b) arc[start angle=45, end angle=225, radius=0.175];
\support{3}{n1}[270];
\support{3}{n4}[270];
\begin{scope}[color=red]
\draw[red,fill=red] (n2) circle (.75ex);
\notation{1}{n2}{$m=10$~kg}[below left];
\draw[red,fill=red] (n6) circle (.75ex);
\notation{1}{n6}{$m=10$~kg}[above right];
\end{scope}
\begin{scope}[color=blue]
\load{1}{n2}[-49.5][.75];
\notation{1}{n2}{$f_{1,\A} \cos(280\pi t)$}[below right=6.5mm];
\point{l}{1.25}{0};
\point{m}{1.25}{-0.275};
\addon{3}{n2}{l}{m}[-1];
\notation{1}{l}{$\alpha$}[below=-0.5mm];
\load{1}{n6}[-49.5][.75];
\notation{1}{n6}{$f_{2,\A} \cos(280\pi t)$}[below right=6.5mm];
\point{l2}{2.25}{1};
\point{m2}{2.25}{0.675};
\addon{3}{n6}{l2}{m2}[-1];
\notation{1}{l2}{$\beta$}[below=-0.5mm];
\end{scope}
\dimensioning{1}{n1}{n2}{-1.45}[$1$~m];
\dimensioning{1}{n2}{n3}{-1.45}[$1$~m];
\dimensioning{2}{n1}{n4}{-0.75}[$1$~m];
\end{tikzpicture}
\caption{$10$-segment frame problem: design domain and boundary conditions. For the fundamental free-vibration eigenvalue constraint, we additionally include the nonstructural mass shown in red, while for the dynamic compliance and peak input power scenarios the uniform-phase forces are drawn in blue.}
\label{fig:frame20}
\end{figure}
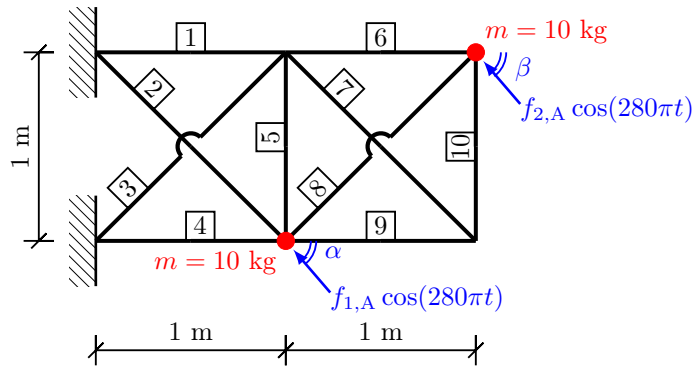

The design domain is a two-dimensional frame structure with a total length of $2$~m and a height of $1$~m that is clamped at the left side, see Fig.~\ref{fig:frame20}. The design variables are the cross-sectional areas of the segments, which are denoted by $\mathbf{a} = [a_1, \dots, a_{10}]\T$. We model each segment of the structure by two Euler--Bernoulli elements of equal length and assign them the same cross-sectional variable. For all elements, we assume a circular cross section made of aluminum (of Young modulus $E=68.9$~GPa and density $\rho=2,770$~kg/m$^3$). For optimization, we bound the fundamental free-vibration eigenfrequency by $140$~Hz from below.

In what follows, we investigate three equivalent settings of the problem: free-vibration, peak input power, and dynamic compliance constraints. The first two scenarios are subject to nonstructural masses of $10$~kg (indicated with solid red circles in Fig.~\ref{fig:frame20}), while the remaining ones are subject to harmonic loads (drawn in blue in Fig.~\ref{fig:frame20}) of angular frequency $\omega=280\pi$~rad/s and unknown worst-case amplitudes $f_{1,\A}$, $f_{2,\A}$ such that $f_{1,\A}^2 + f_{2,\A}^2 \le 100^2$ N$^2$.

\subsubsection{Free-vibration eigenvalue constraint}

As the first optimization variant, we investigate the free-vibration setting of \citet{Ni2014}, who placed two nonstructural masses of $10$~kg on the structure, see the red circles in Fig.~\ref{fig:frame20}. For this setting, we obtained an initial feasible design of the weight $264.392$~kg. Using this objective upper bound, we proceeded with solving the moment relaxations of the polynomial optimization problem (recall Section \ref{sec:moment}).

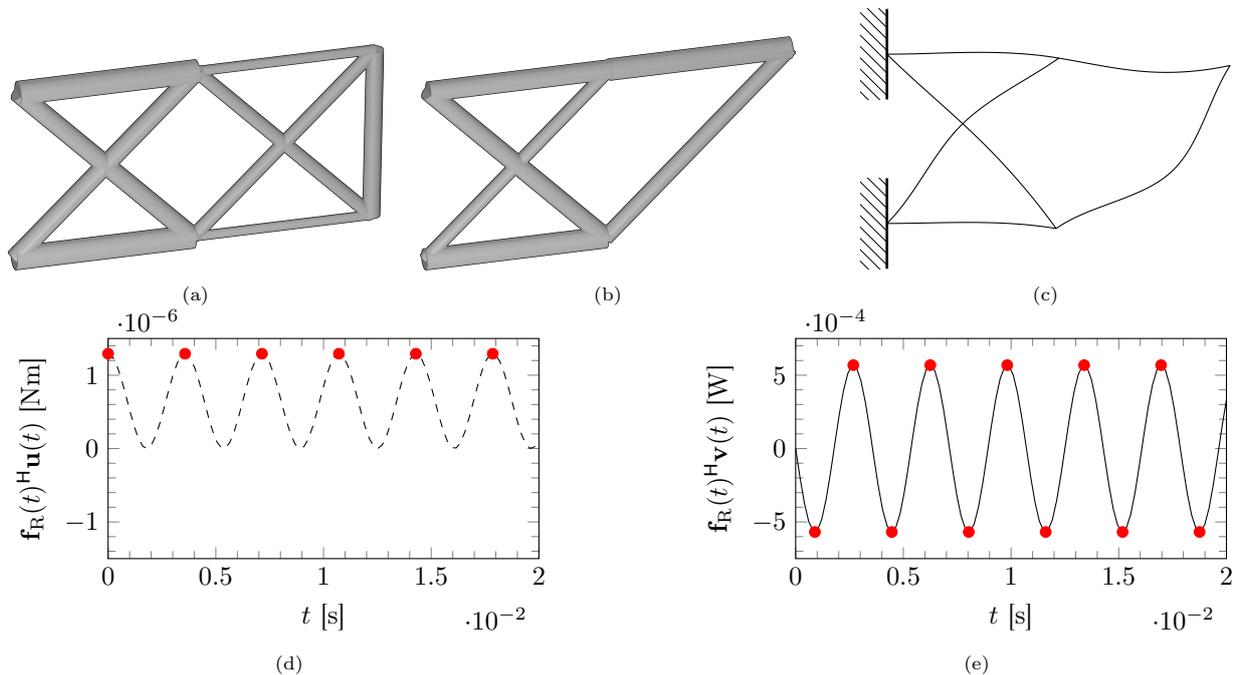
\begin{figure}[!t]
\centering
\begin{subfigure}{0.3\linewidth}
\centering
\includegraphics[width=\linewidth]{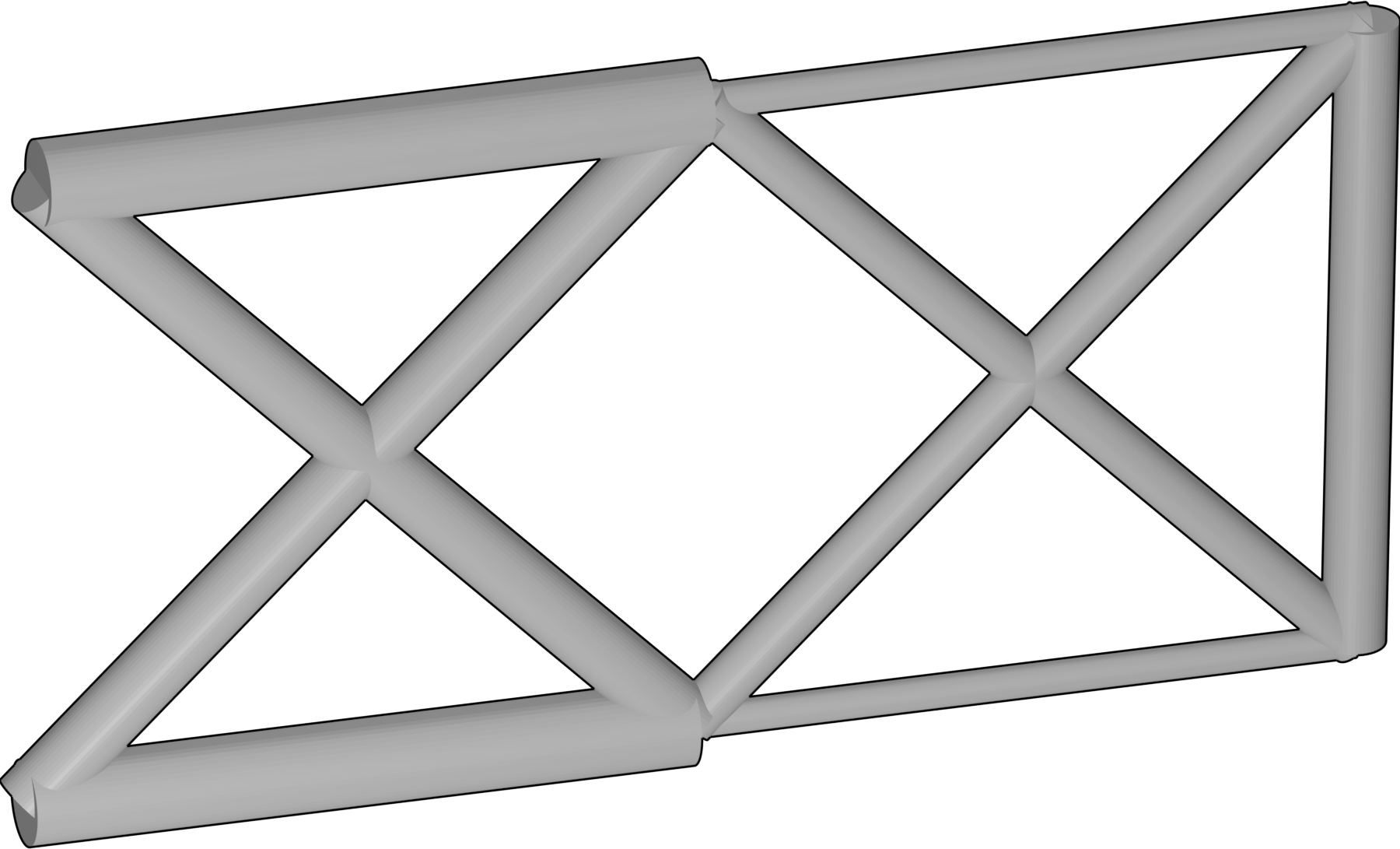}
\caption{}
\label{fig:cantilever20_ni}
\end{subfigure}\hfill%
\begin{subfigure}{0.3\linewidth}
\centering
\includegraphics[width=\linewidth]{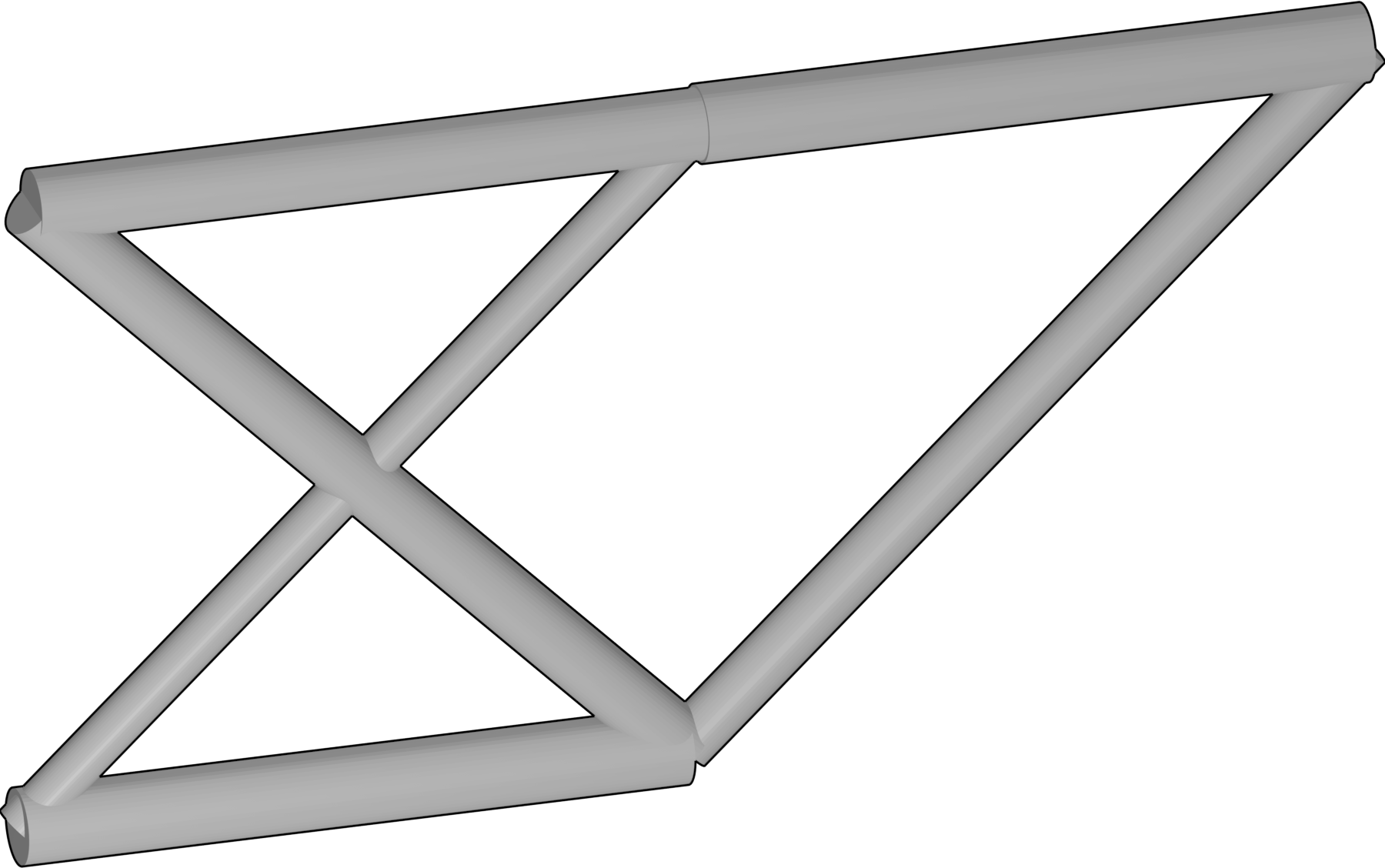}
\caption{}
\label{fig:cantilever20_global}
\end{subfigure}\hfill%
\begin{subfigure}{0.34\linewidth}
\centering
\begin{tikzpicture}
\begin{axis}[legend pos=north west, width=7cm, height=4.5cm, hide axis, unbounded coords=jump, unit vector ratio={1 1}]
\addplot[mark=none] table {ni10vib.dat};
\node (S) at (axis cs:0,0) {};
\node (S2) at (axis cs:0,1) {};
\end{axis}
\support{3}{S}[270];
\support{3}{S2}[270];
\end{tikzpicture}
\caption{}
\label{fig:cantilever20_eigenmode}
\end{subfigure}
\begin{subfigure}{0.45\linewidth}
\begin{tikzpicture}
\begin{axis}[legend style={at={(0.5,0.03)},anchor=south}, height=4.5cm, width=7.25cm, xmin=0, xmax=0.02,ymin=-1.5e-6,ymax=1.5e-6,samples=100,domain=0:0.02,xlabel={$t$ [s]},ylabel={$\mathbf{f}_\R(t)\TC \mathbf{u}(t)$ [Nm]},minor tick num=4]
\addplot+[mark=none, dashed, black] {1.29236203625431e-06*(cos(deg(2*pi*140*x)))^2};    
\addplot[only marks,red] table[row sep=crcr] {
0 1.29236203625431e-06\\
0.00357142857142857 1.29236203625431e-06\\
0.00714285714285714 1.29236203625431e-06\\
0.0107142857142857 1.29236203625431e-06\\
0.0142857142857143 1.29236203625431e-06\\
0.0178571428571429 1.29236203625431e-06\\
};
\end{axis}
\end{tikzpicture}
\caption{}
\label{fig:cantilever20_compliance}
\end{subfigure}%
\hfill\begin{subfigure}{0.45\linewidth}
\begin{tikzpicture}
\begin{axis}[legend style={at={(0.5,0.03)},anchor=south}, height=4.5cm, width=7.25cm, xmin=0, xmax=0.02, ymin=-0.75e-3, ymax=0.75e-3, samples=100,domain=0:0.02,xlabel={$t$ [s]},ylabel={$\mathbf{f}_\R(t)\TC \mathbf{v}(t)$ [W]},minor tick num=4]
\addplot+[mark=none,black] {0.000568410511042483*(-sin(deg(2*2*pi*140*x)))};    
\addplot[only marks,red] table[row sep=crcr] {
0.000892857142857143 -0.000568410511042483\\
0.00267857142857143 0.000568410511042483\\
0.00446428571428571 -0.000568410511042483\\
0.00625 0.000568410511042483\\
0.00803571428571428 -0.000568410511042483\\
0.00982142857142857 0.000568410511042483\\
0.0116071428571429 -0.000568410511042483\\
0.0133928571428571 0.000568410511042483\\
0.0151785714285714 -0.000568410511042483\\
0.0169642857142857 0.000568410511042483\\
0.01875 -0.000568410511042483\\
};
\end{axis}
\end{tikzpicture}
\caption{}
\label{fig:cantilever20_power}
\end{subfigure}
\caption{$10$-segment frame problem: (a) Best design reported in \citep{Ni2014} ($195.59$~kg), (b) global minimum weight design ($w^*=148.44$~kg) and the (c) corresponding eigenmode associated with fundamental free-vibration eigenvalue which concurrently constitutes the scaled amplitudes of worst-case displacements in the robust dynamic compliance constraint and of worst-case velocities in the robust peak input power constraint. For the optimal design, we further show (d) instantaneous force-displacement product, with peaks at the maxima corresponding to dynamic compliance values, and (e) instantaneous force-velocity product, where both maxima and minima attain the same absolute value, corresponding to peak input power. Note that the peak values in (d) and (e) are related by $d_\R = 2p_\R/\omega$, but the time instants of their occurrence differ.}
\label{fig:cantilever20}
\end{figure}

In the first relaxation of the hierarchy, we obtain a lower bound of $\underline{w}^{(1)} = 33.978$~kg and an upper bound of $\hat{w}^{(1)} = 148.442$~kg, which already provides the globally-optimal structure in Fig.~\ref{fig:cantilever20_global} and represents a significantly lighter design compared to the $195.592$~kg structure reported in \citep{Ni2014}, see Fig.~\ref{fig:cantilever20_ni}. Yet, the global optimality is currently not certified and additional degrees of the hierarchy need to be taken. In particular, the second relaxation improves the lower bound to $\underline{w}^{(2)} = 81.849$~kg, while the third one to $\underline{w}^{(3)} = 148.106$~kg, making the optimality gap as small as $0.336$~kg ($0.23\%$ of the design weight). Final, potential reductions in the optimality gap can be achieved in the fourth ($\underline{w}^{(4)}=148.420$ kg) and the fifth ($\underline{w}^{(5)} = 148.438$ kg) relaxations; see Table~\ref{tab:frame20_fv}.

The globally-optimal frame in Fig.~\ref{fig:cantilever20_global} is formed by $6$ segments: $a_1 = 88.285$ cm$^2$, $a_2 = 76.443$ cm$^2$, $a_3 = 37.742$ cm$^2$, $a_4 = 102.133$ cm$^2$, $a_6 = 105.568$ cm$^2$, and $a_8 = 55.453$ cm$^2$, while the remaining segments are removed from the structure. The fundamental eigenmode in Fig.~\ref{fig:cantilever20_eigenmode} predominantly exhibits bending of the beams at the right, and occurs at the natural frequency of $140$~Hz.

\subsubsection{Robust dynamic compliance and peak input power constraints}

As an alternative but equivalent problem, we consider weight minimizations under robust dynamic compliance and robust peak input power constraints. In both of these cases, we introduce uniform-phase loads of the angular frequency $\omega = 280 \pi$ rad/s and search for their worst-case combination such that $f_{1,\A}^2 + f_{2,\A}^2 \le 100^2$ N$^2$. In particular, we restrict the dynamic compliances of all such loads from above by $\overline{d}_\R = 1/(78400 \pi^2)$ Nm, whereas $\overline{p}_\R = 1/(560 \pi)$ W is the upper limit for all possible peak input powers.

Using the constraints \eqref{eq:robustdynamiccompliance} and \eqref{eq:robust_peakpower}, we again obtain an initial feasible design of the weight $\overline{w} = 264.392$~kg. The moment relaxations of both problems yield almost the same lower bounds as in the free-vibration case, see Tables~\ref{tab:frame20_dc} and \ref{tab:frame20_pp}, with the difference caused by numerical inaccuracies. The upper bounds are also very similar, with the best design in Fig.~\ref{fig:cantilever20_global} feasible for both scenarios, thus confirming the equivalence of all three problem formulations. Interestingly, the smallest optimality gap $\varepsilon_\R = 10^{-6}$ for the optimal design follows from the peak input power case.

The worst-case loads for the dynamic compliance and peak input power problems are identical: $f_{1,\A} = 41.617$~N and $f_{2,\A} = 90.929$~N, which correspond to the load angles $\alpha = 74.98^\circ$ and $\beta = 100.21^\circ$, see Fig.~\ref{fig:frame20}. The recovered worst-case amplitudes of displacement/velocity fields are aligned with the first eigenmode, confirming the theoretical relations of Section~\ref{sec:relation}. Finally, the worst-case instantaneous force-displacement and force-velocity products are shown in Figs.~\ref{fig:cantilever20_compliance} and \ref{fig:cantilever20_power}, respectively, and confirm that the robust constraints are active at the optimum.

\begin{table}[!t]
\centering
\begin{subtable}{\linewidth}
\centering
\begin{tabular}{cccccccc}
$r$ & $\underline{w}$ & $\overline{w}_{\min}$ & $\varepsilon_\R$ & $n_\mathrm{c} \times m$ & $n$ & $t$ [s]\\
\hline 1 & $33.978$ & $148.442$ & $3.37$ & $11\times 1$, $1\times11$, $1\times 42$ & $65$ & $0.09$ \\
2 & $81.849$ & $148.442$ & $0.81$ & $11\times 11$, $1\times 21$, $1\times 462$ & $790$ & $9.00$\\
3 & $148.106$ & $148.442$ & $2.3\times 10^{-3}$ & $11\times21$, $1\times31$, $1\times882$ & $1605$ & $83.23$\\
4 & $148.420$ & $148.442$ & $1.5\times 10^{-4}$ & $11\times31$, $1\times41$, $1\times1302$ & $4370$ & $308.91$\\
5 & $148.438$ & $148.442$ & $2.9\times 10^{-5}$ & $11\times41$, $1\times51$, $1\times1722$ & $8665$ & $501.92$
\end{tabular}
\caption{}
\label{tab:frame20_fv}
\end{subtable}\\
\begin{subtable}{\linewidth}
\centering
\begin{tabular}{cccccccc}
$r$ & $\underline{w}$ & $\overline{w}_{\min}$ & $\varepsilon_\R$ & $n_\mathrm{c} \times m$ & $n$ & $t$ [s]\\
\hline 1 & $33.978$ & $148.442$ & $3.37$ & $11\times 1$, $1\times11$, $1\times 46$ & $65$ & $0.08$ \\
2 & $81.847$ & $148.442$ & $0.81$ & $11\times 11$, $1\times 21$, $1\times 506$ & $790$ & $10.99$\\
3 & $148.162$ & $148.442$ & $1.9\times 10^{-3}$ & $11\times21$, $1\times31$, $1\times966$ & $1605$ & $506.65$\\
4 & $148.423$ & $148.442$ & $1.3\times 10^{-4}$ & $11\times31$, $1\times41$, $1\times1426$ & $4370$ & $291.04$\\
5 & $148.436$ & $148.442$ & $3.7\times 10^{-5}$ & $11\times41$, $1\times51$, $1\times1886$ & $8665$ & $563.85$
\end{tabular}
\caption{}
\label{tab:frame20_dc}
\end{subtable}\\
\begin{subtable}{\linewidth}
\centering
\begin{tabular}{cccccccc}
$r$ & $\underline{w}$ & $\overline{w}_{\min}$ & $\varepsilon_\R$ & $n_\mathrm{c} \times m$ & $n$ & $t$ [s]\\
\hline 1 & $33.979$ & $148.442$ & $3.37$ & $11\times 1$, $1\times11$, $1\times 46$ & $65$ & $0.08$ \\
2 & $81.853$ & $148.442$ & $0.81$ & $11\times 11$, $1\times 21$, $1\times 506$ & $790$ & $12.19$\\
3 & $148.134$ & $148.442$ & $2.1\times 10^{-3}$ & $11\times21$, $1\times31$, $1\times966$ & $1605$ & $218.35$\\
4 & $148.402$ & $148.442$ & $2.7\times 10^{-4}$ & $11\times31$, $1\times41$, $1\times1426$ & $4370$ & $313.86$\\
5 & $148.442$ & $148.442$ & $1.0\times 10^{-6}$ & $11\times41$, $1\times51$, $1\times1886$ & $8665$ & $552.93$
\end{tabular}
\caption{}
\label{tab:frame20_pp}
\end{subtable}
\caption{$10$-segment frame problem: (a) free-vibration, (b) dynamic compliance, and (c) peak input power variants. $r$ denotes the relaxation degree, $\underline{w}$ the relaxation lower bound, $\overline{w}_{\min}$ stands for the best upper bound found so far, $\varepsilon_\R$ is the relative optimality gap, $n_\mathrm{c} \times m$ provides the numbers $n_\mathrm{c}$ of semidefinite constraints of sizes $\mathbb{S}^m$, $n$ is the number of variables, and $t$ is the solution run time.}
\end{table}

\subsection{35-segment frame tower}

As the second benchmark, we investigate the $35$-segment frame tower shown in Fig. \ref{fig:frame35}, originally also introduced in \citep{Ni2014}. The structure is clamped at the bottom, and each segment is discretized by two Euler--Bernoulli elements. All beams are made of aluminum with Young's modulus $E=68.9$~GPa and the density of $\rho=2,770$~kg/m$^3$.

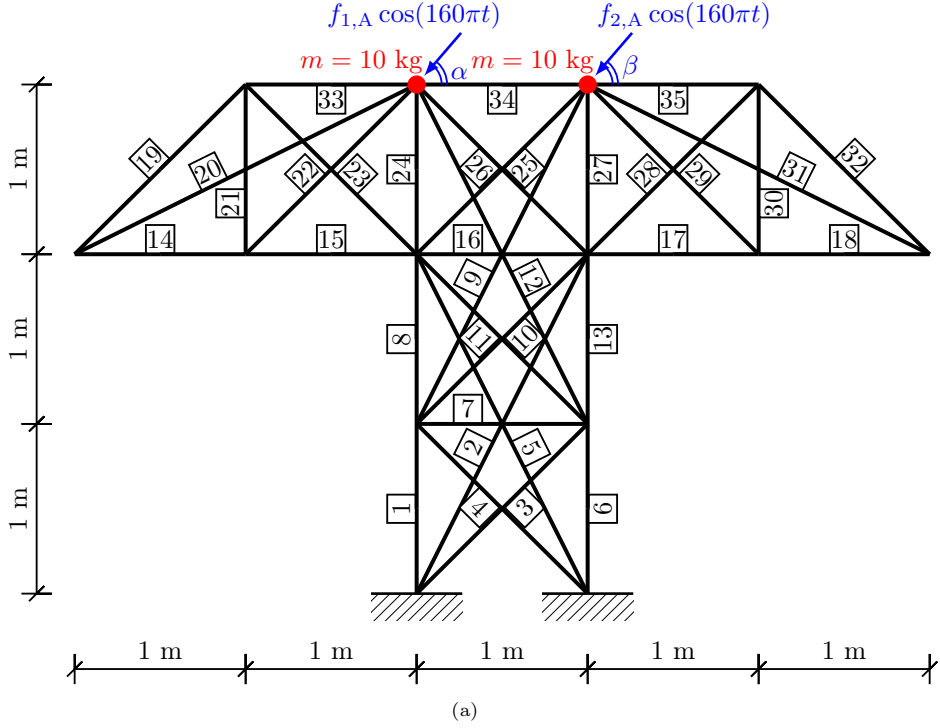
\begin{figure}[!t]
  \centering
		\begin{subfigure}[t]{0.8\linewidth}
      \centering
        \scalebox{1}{
					\begin{tikzpicture}
						\scaling{2.25};
						\point{n1}{0.000000}{0.000000}
						\point{n2}{1.000000}{0.000000}
						\point{n3}{0.000000}{1.000000}
						\point{n4}{1.000000}{1.000000}
						\point{n5}{-2.000000}{2.000000}
						\point{n6}{-1.000000}{2.000000}
						\point{n7}{0.000000}{2.000000}
						\point{n8}{1.000000}{2.000000}
						\point{n9}{2.000000}{2.000000}
						\point{n10}{3.000000}{2.000000}
						\point{n11}{-1.000000}{3.000000}
						\point{n12}{0.000000}{3.000000}
						\point{n13}{1.000000}{3.000000}
						\point{n14}{2.000000}{3.000000}
						\point{n15}{0.000000}{0.500000}
						\point{n16}{0.500000}{1.000000}
						\point{n17}{0.500000}{0.500000}
						\point{n18}{0.500000}{0.500000}
						\point{n19}{0.500000}{1.000000}
						\point{n20}{1.000000}{0.500000}
						\point{n21}{0.500000}{1.000000}
						\point{n22}{0.000000}{1.500000}
						\point{n23}{0.500000}{2.000000}
						\point{n24}{0.500000}{1.500000}
						\point{n25}{0.500000}{1.500000}
						\point{n26}{0.500000}{2.000000}
						\point{n27}{1.000000}{1.500000}
						\point{n28}{-1.500000}{2.000000}
						\point{n29}{-0.500000}{2.000000}
						\point{n30}{0.500000}{2.000000}
						\point{n31}{1.500000}{2.000000}
						\point{n32}{2.500000}{2.000000}
						\point{n33}{-1.500000}{2.500000}
						\point{n34}{-1.000000}{2.500000}
						\point{n35}{-1.000000}{2.500000}
						\point{n36}{-0.500000}{2.500000}
						\point{n37}{-0.500000}{2.500000}
						\point{n38}{0.000000}{2.500000}
						\point{n39}{0.500000}{2.500000}
						\point{n40}{0.500000}{2.500000}
						\point{n41}{1.000000}{2.500000}
						\point{n42}{1.500000}{2.500000}
						\point{n43}{1.500000}{2.500000}
						\point{n44}{2.000000}{2.500000}
						\point{n45}{2.000000}{2.500000}
						\point{n46}{2.500000}{2.500000}
						\point{n47}{-0.500000}{3.000000}
						\point{n48}{0.500000}{3.000000}
						\point{n49}{1.500000}{3.000000}
						\beam{2}{n1}{n15}; \notation{4}{n1}{n3}[$1$];
						\beam{2}{n3}{n15}
						\beam{2}{n1}{n16}; \notation{4}{n1}{n8}[$2$][0.41];
						\beam{2}{n8}{n16}
						\beam{2}{n1}{n17}; \notation{4}{n1}{n4}[$3$][0.57][below];
						\beam{2}{n4}{n17}
						\beam{2}{n2}{n18}; \notation{4}{n2}{n3}[$4$][0.57][below];
						\beam{2}{n3}{n18}
						\beam{2}{n2}{n19}; \notation{4}{n2}{n7}[$5$][0.41];
						\beam{2}{n7}{n19}
						\beam{2}{n2}{n20}; \notation{4}{n2}{n4}[$6$][0.5][below];
						\beam{2}{n4}{n20}
						\beam{2}{n3}{n21}; \notation{4}{n3}{n4}[$7$][0.3];
						\beam{2}{n4}{n21}
						\beam{2}{n3}{n22}; \notation{4}{n3}{n7}[$8$];
						\beam{2}{n7}{n22}
						\beam{2}{n3}{n23}; \notation{4}{n3}{n13}[$9$][0.41];
						\beam{2}{n13}{n23}
						\beam{2}{n3}{n24}; \notation{4}{n3}{n8}[$10$][0.57][below];
						\beam{2}{n8}{n24}
						\beam{2}{n4}{n25}; \notation{4}{n4}{n7}[$11$][0.57][below];
						\beam{2}{n7}{n25}
						\beam{2}{n4}{n26}; \notation{4}{n4}{n12}[$12$][0.41];
						\beam{2}{n12}{n26}
						\beam{2}{n4}{n27}; \notation{4}{n4}{n8}[$13$][0.5][below];
						\beam{2}{n8}{n27}
						\beam{2}{n5}{n28}; \notation{4}{n5}{n6}[$14$];
						\beam{2}{n6}{n28}
						\beam{2}{n6}{n29}; \notation{4}{n6}{n7}[$15$];
						\beam{2}{n7}{n29}
						\beam{2}{n7}{n30}; \notation{4}{n7}{n8}[$16$][0.3];
						\beam{2}{n8}{n30}
						\beam{2}{n8}{n31}; \notation{4}{n8}{n9}[$17$];
						\beam{2}{n9}{n31}
						\beam{2}{n9}{n32}; \notation{4}{n9}{n10}[$18$];
						\beam{2}{n10}{n32}
						\beam{2}{n5}{n33}; \notation{4}{n5}{n11}[$19$];
						\beam{2}{n11}{n33}
						\beam{2}{n5}{n34}; \notation{4}{n5}{n12}[$20$][0.41];
						\beam{2}{n12}{n34}
						\beam{2}{n6}{n35}; \notation{4}{n6}{n11}[$21$][0.3];
						\beam{2}{n11}{n35}
						\beam{2}{n6}{n36}; \notation{4}{n6}{n12}[$22$][0.41];
						\beam{2}{n12}{n36}
						\beam{2}{n7}{n37}; \notation{4}{n7}{n11}[$23$][0.41];
						\beam{2}{n11}{n37}
						\beam{2}{n7}{n38}; \notation{4}{n7}{n12}[$24$];
						\beam{2}{n12}{n38}
						\beam{2}{n7}{n39}; \notation{4}{n7}{n13}[$25$][0.57][below];
						\beam{2}{n13}{n39}
						\beam{2}{n8}{n40}; \notation{4}{n8}{n12}[$26$][0.57][below];
						\beam{2}{n12}{n40}
						\beam{2}{n8}{n41}; \notation{4}{n8}{n13}[$27$][0.5][below];
						\beam{2}{n13}{n41}
						\beam{2}{n8}{n42}; \notation{4}{n8}{n14}[$28$][0.41];
						\beam{2}{n14}{n42}
						\beam{2}{n9}{n43}; \notation{4}{n9}{n13}[$29$][0.41];
						\beam{2}{n13}{n43}
						\beam{2}{n9}{n44}; \notation{4}{n9}{n14}[$30$][0.3][below];
						\beam{2}{n14}{n44}
						\beam{2}{n10}{n45}; \notation{4}{n10}{n13}[$31$][0.41];
						\beam{2}{n13}{n45}
						\beam{2}{n10}{n46}; \notation{4}{n10}{n14}[$32$];
						\beam{2}{n14}{n46}
						\beam{2}{n11}{n47}; \notation{4}{n11}{n12}[$33$][0.5][below];
						\beam{2}{n12}{n47}
						\beam{2}{n12}{n48}; \notation{4}{n12}{n13}[$34$][0.5][below];
						\beam{2}{n13}{n48}
						\beam{2}{n13}{n49}; \notation{4}{n13}{n14}[$35$][0.5][below];
						\beam{2}{n14}{n49}
				
						\support{3}{n1}[0];
						\support{3}{n2}[0];
						\begin{scope}[color=red]
							\draw[red,fill=red] (n12) circle (.75ex);
							\notation{1}{n12}{$m=10$~kg}[above left=0.3mm and -2mm];
							\draw[red,fill=red] (n13) circle (.75ex);
							\notation{1}{n13}{$m=10$~kg}[above left=0.3mm and -2mm];
						\end{scope}
						\begin{scope}[color=blue]
							\load{1}{n12}[49.5][.75];
							\notation{1}{n12}{$f_{1,\A} \cos(160\pi t)$}[above right=6mm and -13mm];
							\point{l}{0.25}{3};
							\point{m}{0.25}{3.275};
							\addon{3}{n12}{m}{l}[-1];
							\notation{1}{l}{$\alpha$}[above=-0.5mm];
							\load{1}{n13}[49.5][.75];
							\notation{1}{n13}{$f_{2,\A} \cos(160\pi t)$}[above right=6mm and 0mm];
							\point{l2}{1.25}{3};
							\point{m2}{1.25}{3.275};
							\addon{3}{n13}{m2}{l2}[-1];
							\notation{1}{l2}{$\beta$}[above=-0.5mm];
						\end{scope}
						\dimensioning{1}{n5}{n6}{-1.0}[$1$~m];
						\dimensioning{1}{n6}{n7}{-1.0}[$1$~m];
						\dimensioning{1}{n7}{n8}{-1.0}[$1$~m];
						\dimensioning{1}{n8}{n9}{-1.0}[$1$~m];
						\dimensioning{1}{n9}{n10}{-1.0}[$1$~m];
						\dimensioning{2}{n1}{n4}{-5.0}[$1$~m];
						\dimensioning{2}{n4}{n8}{-5.0}[$1$~m];
						\dimensioning{2}{n8}{n13}{-5.0}[$1$~m];
					\end{tikzpicture}}
		\caption{}
		\end{subfigure}%
\caption{$35$-segment frame problem: design domain and boundary conditions. For the fundamental free-vibration eigenvalue constraint, we additionally include the nonstructural mass shown in red, while for the dynamic compliance and peak input power scenarios the uniform-phase forces are drawn in blue.}
\label{fig:frame35}
\end{figure}

Due to the equivalence of constraints and similar numerical performance observed in the previous example, recall Tables \ref{tab:frame20_fv}--\ref{tab:frame20_pp}, here we explicitly solve only the free-vibration-eigenvalue-based formulation. However, we interpret the designs also with respect to the other two formulations. 

For free-vibrations, the structure carries two nonstructural masses of $10$~kg each placed at the central top nodes, see the solid red circles in Fig. \ref{fig:frame35}. The optimization objective is to minimize the total structural weight while enforcing a lower bound of $f = 80$~Hz for the nonzero natural frequencies. Analogously, for the robust dynamic compliance and robust peak input power versions, we consider harmonic loads of angular frequency $\omega = 160 \pi$~rad/s and unknown worst-case amplitudes $f_{1,\A}$, $f_{2,\A}$ such that $f_{1,\A}^2 + f_{2,\A}^2 \le 100^2$ N$^2$. In the dynamic compliance case, the instantaneous product of all such forces and corresponding displacements is bounded from above by $\overline{d}_\R = 1/(25,600 \pi^2)$ Nm, while in the peak input power case, the instantaneous product of all these forces and associated velocities is limited by $\overline{p}_\R = 1/(320 \pi)$ W. 

\begin{table}[!b]
  \centering
    \begin{tabular}{lrrrrrrr}
		$r$ & $\underline{w}$ & $\overline{w}_{\min}$ & $\varepsilon_\R$ & $n_\mathrm{c}\times m$ & $n$ & $t$ [s]\\
		\hline
		$1$ & {$28.02$} & {$85.82$} & {$2.06$} & $36\times1$, $1\times 36$, $1\times141$ & $665$ & {$0.25$} \\
		$2$ & {$38.02$} & {$70.92$} & {$0.87$} & $36\times36$, $1\times 71$, $1\times5076$ & $29\thinspace890$ & {$16\thinspace636$} \\
		\end{tabular}
    \caption{$35$-segment frame problem and its free-vibration settings: $r$ denotes the relaxation degree, $\underline{w}$ the relaxation lower bound, $\overline{w}_{\min}$ stands for the best upper bound found so far, $\varepsilon_\R$ is the relative optimality gap, $n_\mathrm{c} \times m$ provides the numbers $n_\mathrm{c}$ of semidefinite constraints of sizes $\mathbb{S}^m$, $n$ is the number of variables, and $t$ is the solution run time.}
    \label{tab:frame35}
\end{table}

The optimization problem is already too large to be solved to certified global optimality. However, we can still provide high-quality designs with guaranteed optimality gaps. For the free-vibration version, we obtain an initial feasible design of the weight $\overline{w} = 1,323.25$~kg. In the first relaxation of the hierarchy, we receive a lower bound of $\underline{w}^{(1)} = 28.02$~kg and an upper bound of $\hat{w}^{(1)} = 85.82$~kg. This upper bound represents a significantly lighter design compared to $2,205$~kg that was reported in \citep{Ni2014}, see also the design visualization in Fig.~\ref{fig:frame35_ni}. The second relaxation improves the lower bound to $\underline{w}^{(2)} = 38.05$~kg, and the upper bound to $\hat{w}^{(2)} = 70.92$~kg, with the related design visualized in Fig.~\ref{fig:frame35_best}. 
For this design, we have $a_1=37.236$~cm$^2$, $a_7 = 7.656$~cm$^2$, $a_8 = 33.033$~cm$^2$, $a_{10}=13.908$~cm$^2$, $a_{11} = 16.899$~cm$^2$, $a_{13} = 40.102$~cm$^2$, $a_{16} = 2.987$~cm$^2$, $a_{24} = 7.192$~cm$^2$, $a_{25} = 15.538$~cm$^2$, $a_{26} = 17.556$, $a_{27} = 6.745$~cm$^2$, and $a_{34} = 30.724$~cm$^2$, while all other cross-section areas are zero.
Although the optimality gap remains large, the designs obtained are high quality local solutions that are significantly lighter than those reported in \citep{Ni2014}, showing that the method can handle nonconvexity in larger problems more effectively than previously possible.

The optimized design (Fig.~\ref{fig:frame35_best}) exhibits a clear load path from the top to the base. Its fundamental eigenfrequency of $80$ Hz is repeated (Fig.~\ref{fig:frame35_eigenmode}), with the eigenmodes showing a bending-dominated deformation pattern. Finally, we interpret the optimized design obtained with respect to the robust dynamic compliance and peak input power formulations. For the first natural mode, we recover the worst-case load as $\alpha = 24.11^\circ$ and $\beta = -2.82^\circ$ together with $f_{\A,1} = 74.84$~N and $f_{\A,2} = 66.33$~N, while for the second natural mode, we obtain $\alpha = 93.95^\circ$ and $\beta = 102.85^\circ$ together with $f_{\A,1} = 57.06$~N and $f_{\A,2} = 82.12$~N; see Fig.~\ref{fig:frame35}. Finally, the worst-case instantaneous force-displacement and force-velocity products are shown in Figs.~\ref{fig:frame35_compliance} and \ref{fig:frame35_power}, respectively, and confirm that the robust constraints are again active at the optimum.

\begin{figure}[!t]
\centering
\begin{subfigure}{0.3\linewidth}
  \includegraphics[height=3cm]{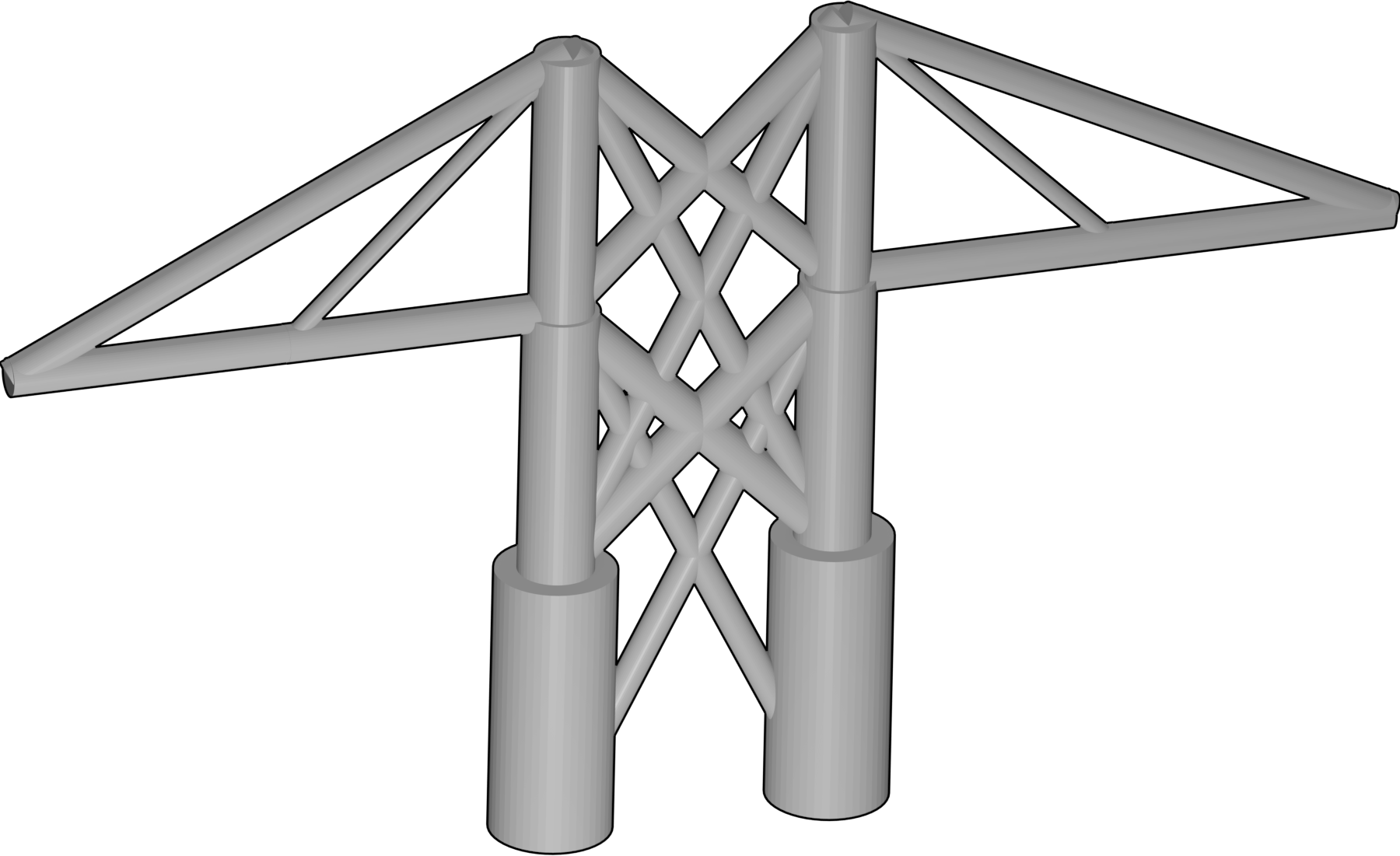}
  \caption{}
  \label{fig:frame35_ni}
\end{subfigure}%
\begin{subfigure}{0.3\linewidth}
	\centering
	\begin{tikzpicture}
			\begin{axis}[legend pos=north west, width=7cm, height=5.6cm, hide axis, unbounded coords=jump, unit vector ratio={1 1}]
					\addplot[mark=none] table {ni35orig.dat};
					\node (S) at (axis cs:0,0) {};
					\node (S2) at (axis cs:1,0) {};
				\end{axis}
			\support{3}{S}[0];
			\support{3}{S2}[0];
		\end{tikzpicture}
	\caption{}
\end{subfigure}%
\begin{subfigure}{0.1\linewidth}
  \centering
  \includegraphics[height=3cm]{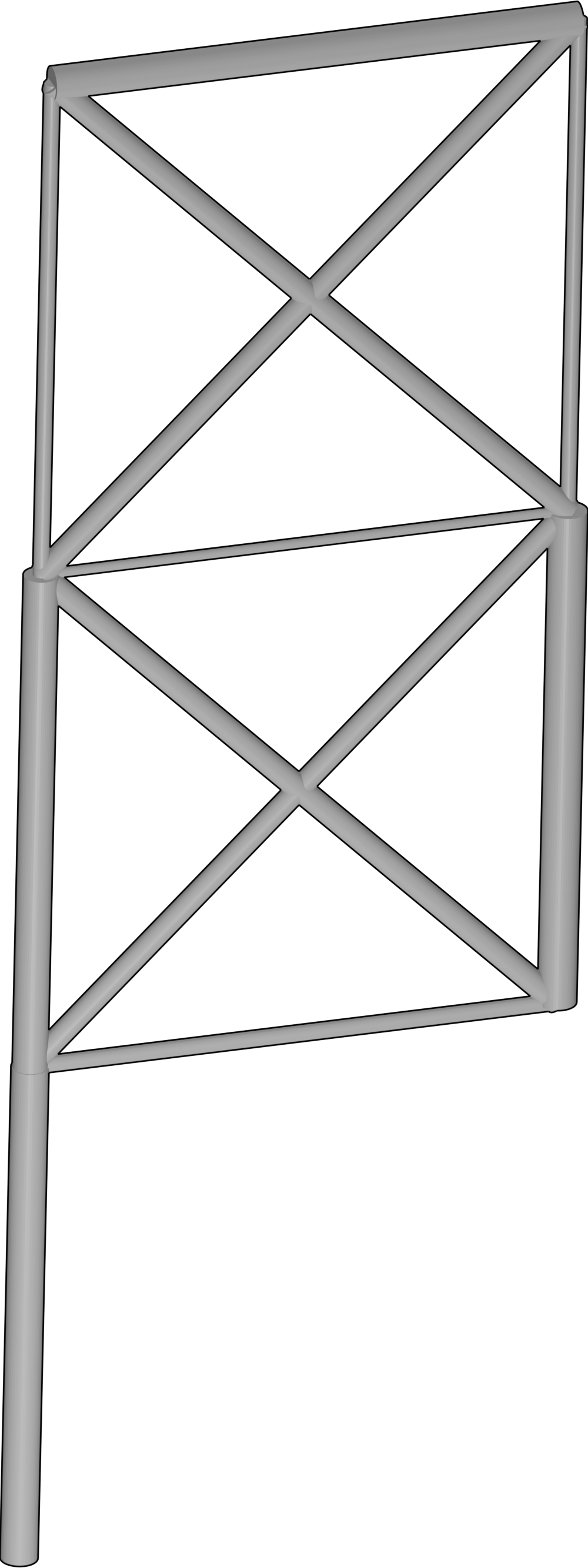}
  \caption{}
  \label{fig:frame35_best}
\end{subfigure}%
\begin{subfigure}{0.3\linewidth}
	\centering
	\begin{tikzpicture}
			\begin{axis}[legend pos=north west, width=7cm, height=4.9cm, hide axis, unbounded coords=jump, unit vector ratio={1 1}]
					\addplot[mark=none] table {ni35vib.dat};
					\node (S) at (axis cs:0,0) {};
					\node (S2) at (axis cs:1,0) {};
				\end{axis}
			\support{3}{S}[0];
			\support{3}{S2}[0];
		\end{tikzpicture}%
  	\begin{tikzpicture}
			\begin{axis}[legend pos=north west, width=7cm, height=4.85cm, hide axis, unbounded coords=jump, unit vector ratio={1 1}]
					\addplot[mark=none] table {ni35vib2.dat};
					\node (S) at (axis cs:0,0) {};
					\node (S2) at (axis cs:1,0) {};
				\end{axis}
			\support{3}{S}[0];
			\support{3}{S2}[0];
		\end{tikzpicture}
	\caption{}
  \label{fig:frame35_eigenmode}
\end{subfigure}\\
\begin{subfigure}{0.45\linewidth}
\begin{tikzpicture}
\begin{axis}[legend style={at={(0.5,0.03)},anchor=south}, height=4.5cm, width=7.25cm, xmin=0, xmax=0.02,ymin=-4.5e-6,ymax=4.5e-6,samples=100,domain=0:0.02,xlabel={$t$ [s]},ylabel={$\mathbf{f}_\R(t)\TC \mathbf{u}(t)$ [Nm]},minor tick num=4]
\addplot+[mark=none, dashed, black] {3.95785873602882e-06*(cos(deg(2*pi*80*x)))^2};    
		\addplot[only marks,red] table[row sep=crcr] {
				0 3.95785873602882e-06\\
				0.00625 3.95785873602882e-06\\
				0.0125 3.95785873602882e-06\\
				0.01875 3.95785873602882e-06\\
				};
\end{axis}
\end{tikzpicture}
\caption{}
\label{fig:frame35_compliance}
\end{subfigure}%
\hfill\begin{subfigure}{0.45\linewidth}
\begin{tikzpicture}
\begin{axis}[legend style={at={(0.5,0.03)},anchor=south}, height=4.5cm, width=7.25cm, xmin=0, xmax=0.02, ymin=-1.3e-3, ymax=1.3e-3, samples=100,domain=0:0.02,xlabel={$t$ [s]},ylabel={$\mathbf{f}_\R(t)\TC \mathbf{v}(t)$ [W]},minor tick num=4]
\addplot+[mark=none,black] {0.000994718394324346*(-sin(deg(2*2*pi*80*x)))};    
		\addplot[only marks,red] table[row sep=crcr] {
			0.0015625 -0.000994718394324346\\
			0.0046875 0.000994718394324346\\
			0.0078125 -0.000994718394324346\\
			0.0109375 0.000994718394324346\\
			0.0140625 -0.000994718394324346\\
			0.0171875 0.000994718394324346\\
			};
\end{axis}
\end{tikzpicture}
\caption{}
\label{fig:frame35_power}
\end{subfigure}
\caption{$35$-segment frame problem: (a) Best design reported in \citep{Ni2014} ($2205$~kg) and (b) the eigenmode for the lowest natural frequency $80.6$~Hz, (c) optimized design ($\hat{w}^{(2)}=70.9$~kg) and (d) the corresponding eigenmodes related to the repeated free-vibration eigenfrequency of $80$~Hz. These concurrently constitute the scaled amplitudes of worst-case displacements in the robust dynamic compliance constraint and worst-case velocities in the robust peak input power constraint. For the optimized design, we further show (e) instantaneous force-displacement product, with peaks at the maxima corresponding to dynamic compliance values, and (f) instantaneous force-velocity product, where both maxima and minima attain the same absolute value, corresponding to peak input power. Note that the peak values in (e) and (f) are related by $d_\R = 2p_\R/\omega$, but the time instants of their occurrence differ.}
\label{fig:frame35design}
\end{figure}
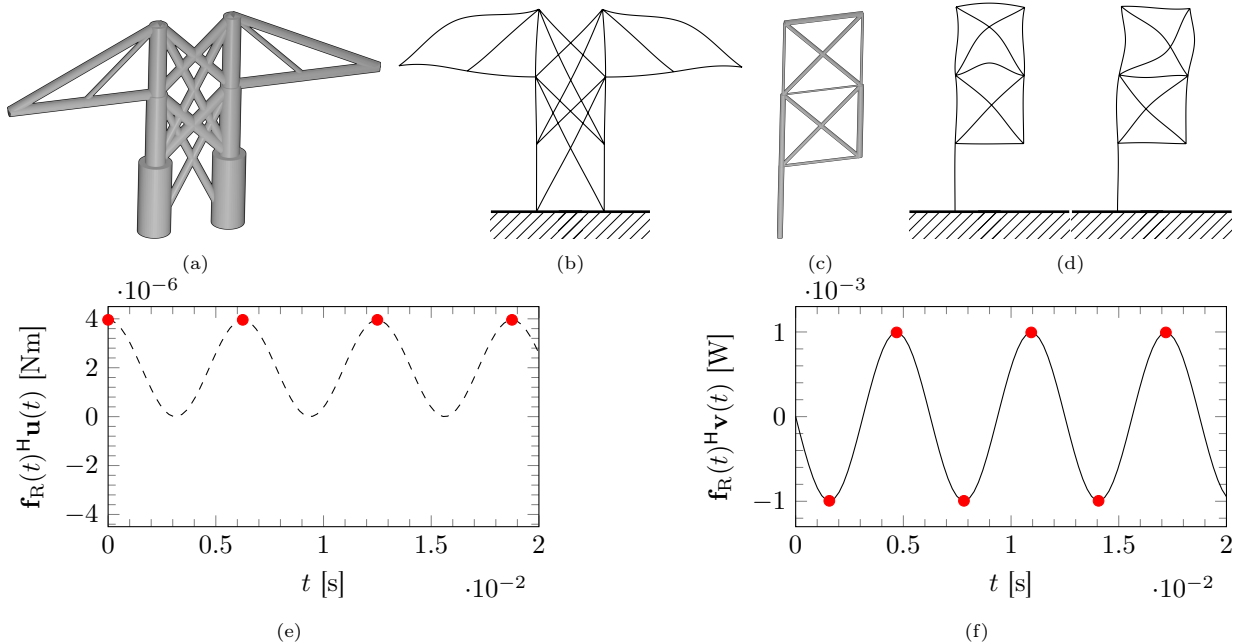

\subsection{Experimental problem}

To demonstrate that the proposed framework can be implemented and validated experimentally, we consider a $12$-segment planar frame structure shown in Fig.~\ref{fig:exp_frame}. The frame is clamped at the left side and each straight segment is discretized by two Euler--Bernoulli beam elements. All beams have a rectangular cross-section with a fixed width of $2$~cm, while the cross-section heights serve as the design variables. Let $h_e$, $e=1,\dots,12$, denote the height of the $e$-th segment. For a rectangular cross-section, the area and second moment of area scale as
\[
a_e = b h_e, 
\qquad 
I_e = \frac{b h_e^3}{12} = \frac{a_e^3}{12 b^2},
\]
where $b=2$~cm denotes the width. As a result, the entries of the global stiffness matrix are polynomial functions of the cross-sectional areas of total degree three, whereas the mass matrix depends linearly on them. The beams are made of Pr\r{u}\v{s}a PETG filament with Young's modulus $E = 1.8$~GPa and density $\rho = 1{,}200$~kg/m$^3$.

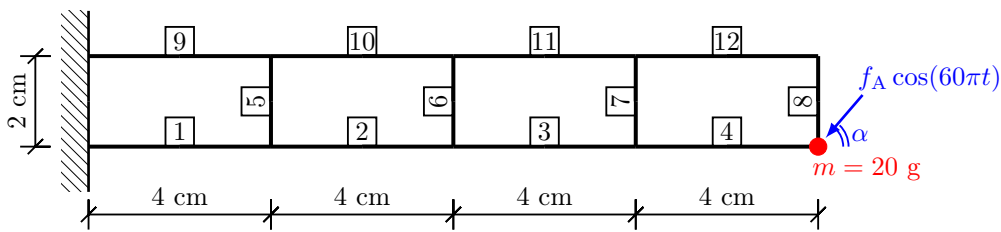
\begin{figure}[!t]
\centering
\begin{tikzpicture}
\point{n1}{0}{0}
\point{n2}{1.2}{0}
\point{n3}{2.4}{0}
\point{n4}{3.6}{0}
\point{n5}{4.8}{0}
\point{n6}{6}{0}
\point{n7}{7.2}{0}
\point{n8}{8.4}{0}
\point{n9}{9.6}{0}
\point{n10}{2.4}{0.6}
\point{n11}{4.8}{0.6}
\point{n12}{7.2}{0.6}
\point{n13}{9.6}{0.6}
\point{n14}{0}{1.2}
\point{n15}{1.2}{1.2}
\point{n16}{2.4}{1.2}
\point{n17}{3.6}{1.2}
\point{n18}{4.8}{1.2}
\point{n19}{6}{1.2}
\point{n20}{7.2}{1.2}
\point{n21}{8.4}{1.2}
\point{n22}{9.6}{1.2}
\beam{2}{n1}{n2}
\notation{4}{n1}{n3}[$1$]
\beam{2}{n2}{n3}
\beam{2}{n3}{n4}
\notation{4}{n3}{n5}[$2$]
\beam{2}{n4}{n5}
\beam{2}{n5}{n6}
\notation{4}{n5}{n7}[$3$]
\beam{2}{n6}{n7}
\beam{2}{n7}{n8}
\notation{4}{n7}{n9}[$4$]
\beam{2}{n8}{n9}
\beam{2}{n3}{n10}
\notation{4}{n3}{n16}[$5$]
\beam{2}{n10}{n16}
\beam{2}{n5}{n11}
\notation{4}{n5}{n18}[$6$]
\beam{2}{n11}{n18}
\beam{2}{n7}{n12}
\notation{4}{n7}{n20}[$7$]
\beam{2}{n12}{n20}
\beam{2}{n9}{n13}
\notation{4}{n9}{n22}[$8$]
\beam{2}{n13}{n22}
\beam{2}{n14}{n15}
\notation{4}{n14}{n16}[$9$]
\beam{2}{n15}{n16}
\beam{2}{n16}{n17}
\notation{4}{n16}{n18}[$10$]
\beam{2}{n17}{n18}
\beam{2}{n18}{n19}
\notation{4}{n18}{n20}[$11$]
\beam{2}{n19}{n20}
\beam{2}{n20}{n21}
\notation{4}{n20}{n22}[$12$]
\beam{2}{n21}{n22}
\begin{scope}[color=red]
\draw[red,fill=red] (n9) circle (.75ex);
\notation{1}{n9}{$m=20$~g}[below right=0.3mm and -2mm];
\end{scope}
\support{3}{n1}[270]
\support{3}{n14}[270]
\point{dx1a}{0}{0}
\point{dx1b}{1.2}{0}
\point{dx2a}{1.2}{0}
\point{dx2b}{2.4}{0}
\dimensioning{1}{dx1a}{dx2b}{-0.9}[$4$~cm]
\point{dx3a}{2.4}{0}
\point{dx3b}{3.6}{0}
\point{dx4a}{3.6}{0}
\point{dx4b}{4.8}{0}
\dimensioning{1}{dx3a}{dx4b}{-0.9}[$4$~cm]
\point{dx5a}{4.8}{0}
\point{dx5b}{6}{0}
\point{dx6a}{6}{0}
\point{dx6b}{7.2}{0}
\dimensioning{1}{dx5a}{dx6b}{-0.9}[$4$~cm]
\point{dx7a}{7.2}{0}
\point{dx7b}{8.4}{0}
\point{dx8a}{8.4}{0}
\point{dx8b}{9.6}{0}
\dimensioning{1}{dx7a}{dx8b}{-0.9}[$4$~cm]
\point{dy1a}{0}{0}
\point{dy1b}{0}{0.6}
\point{dy2a}{0}{0.6}
\point{dy2b}{0}{1.2}
\dimensioning{2}{dy1a}{dy2b}{-0.7}[$2$~cm]
\begin{scope}[color=blue]
	\load{1}{n9}[49.5][.75]; 
	\notation{1}{n9}{$f_{\A} \cos(60\pi t)$}[above right=6mm and 4mm];
	\point{l}{10.15}{0};
  \point{m}{10.15}{0.55};
	\addon{3}{n9}{m}{l}[-1];
	\notation{1}{l}{$\alpha$}[above=-0.5mm];
\end{scope}
\end{tikzpicture}
\caption{Design domain and boundary conditions for the experimental frame problem. The structure is clamped at the left side, and a non-structural mass of $20$~g is placed at the rightmost node (solid black circle). Each segment is discretized by two Euler--Bernoulli elements.}
\label{fig:exp_frame}
\end{figure}

A non-structural mass of $20$~g is attached to the bottom-right node, see Fig.~\ref{fig:exp_frame}. The optimization objective is to minimize the total structural weight while enforcing a lower bound $f_\text{min} = 30$~Hz on all nonzero natural frequencies of free vibration.

The orthogonal layout of the design domain in Fig.~\ref{fig:exp_frame} and the particular value of the frequency bound are chosen to be consistent with the Euler--Bernoulli beam assumptions: the optimized cross-sectional heights remain small compared to the beam lengths, and the element effective lengths are determined by the nodal coordinates rather than by the detailed joint geometry. In addition, we impose symmetry of the design with respect to the horizontal axis. This reduces the number of independent design variables from $12$ to $8$ by enforcing $a_1 = a_9$, $a_2 = a_{10}$, $a_3 = a_{11}$, and $a_4 = a_{12}$.

\subsubsection{Optimization and verification}

We solve the resulting free-vibration eigenvalue-constrained optimization problem using the moment-sum-of-squares hierarchy. The numerical results are summarized in Table~\ref{tab:frame12}. The second relaxation provides a lower bound of $\underline{w}^{(2)} = 1.18$~g and an upper bound of $\hat{w}^{(2)} = 17.04$~g. The third relaxation improves the lower bound to $\underline{w}^{(3)} = 9.07$~g, while the upper bound remains unchanged. The fourth and fifth relaxations yield further small improvements of the lower bound, up to $\underline{w}^{(5)} = 16.41$~g, reducing the relative optimality gap to $\varepsilon_\mathrm{R} = 0.04$. We note that, although higher-order relaxations can in principle be formulated and solved, the Mosek optimizer reported there numerical issues, which prevented us from obtaining reliable results beyond the degree-$5$ relaxation.

\begin{table}[!b]
  \centering
    \begin{tabular}{lrrrrrrr}
		$r$ & $\underline{w}$ & $\overline{w}_{\min}$ & $\varepsilon_\R$ & $n_\mathrm{c}\times m$ & $n$ & $t$ [s]\\
		\hline
		$2$ & {$1.18$} & {$17.04$} & {$13.44$} & $9\times9$, $1\times 17$, $1\times60$ & $424$ & {$0.49$} \\
		$3$ & {$9.07$} & {$17.04$} & {$0.88$} & $9\times17$, $1\times 25$, $1\times540$ & $1\thinspace028$ & {$39.14$} \\
		$4$ & {$14.59$} & {$17.04$} & {$0.17$} & $9\times25$, $1\times 33$, $1\times1020$ & $2\thinspace248$ & {$243.06$} \\
		$5$ & {$16.41$} & {$17.04$} & {$0.04$} & $9\times33$, $1\times 41$, $1\times1500$ & $4\thinspace420$ & {$940.44$} \\
    \end{tabular}
    \caption{$12$-segment frame problem and its free-vibration settings: $r$ denotes the relaxation degree, $\underline{w}$ the relaxation lower bound, $\overline{w}_{\min}$ stands for the best upper bound found so far, $\varepsilon_\R$ is the relative optimality gap, $n_\mathrm{c} \times m$ provides the numbers $n_\mathrm{c}$ of semidefinite constraints of sizes $\mathbb{S}^m$, $n$ is the number of variables, and $t$ is the solution run time.}
    \label{tab:frame12}
\end{table}

The best design obtained in the relaxation procedure is visualized in Fig.~\ref{fig:exp_frame_design}a. The corresponding cross-sectional heights are $h_1 = 1.820$~mm, $h_2 = 1.751$~mm, $h_3 = 1.654$~mm, $h_4 = 1.517$~mm, $h_5 = 2.432$~mm, $h_6 = 2.392$~mm, $h_7 = 2.255$~mm, and $h_8 = 1.449$~mm. For the optimized design, we further investigate the effect of mesh refinement on the free-vibration eigenfrequencies. In particular, for the optimization discretization of two elements per segment, we obtain the first three nonzero natural frequencies as $f_1 = 30.0$~Hz, $f_2 = 126.7$~Hz, and $f_3 = 256.1$~Hz and the eigenmodes visualized in Fig.~\ref{fig:exp_frame_design}b. Refining the mesh to $12$ elements per segment leads to $f_1 = 30.0$~Hz, $f_2 = 126.7$~Hz, and $f_3 = 256.0$~Hz. This indicates that the optimized design is not significantly affected by mesh refinement and that the frequency constraint is accurately enforced.

Similarly to the free-vibration case, we may investigate the robust dynamic compliance and robust peak input power formulations for this problem. In this case, we consider harmonic loads of angular frequency $\omega = 60 \pi$~rad/s and unknown worst-case amplitude $f_{\A}$ such that $\lvert f_{\A}\rvert^2 \le 0.02$~N$^2$. In the dynamic compliance case, the instantaneous product of all such forces and corresponding displacements is bounded from above by $\overline{d}_\R = 1/(3,600 \pi^2)$ Nm, while in the peak input power case, the instantaneous product of all these forces and associated velocities is limited by $\overline{p}_\R = 1/(120 \pi)$ W. For the optimized design, we recover the worst-case load direction as $\alpha = 89.51^\circ$ together with $f_{\A} = 0.141$~N.

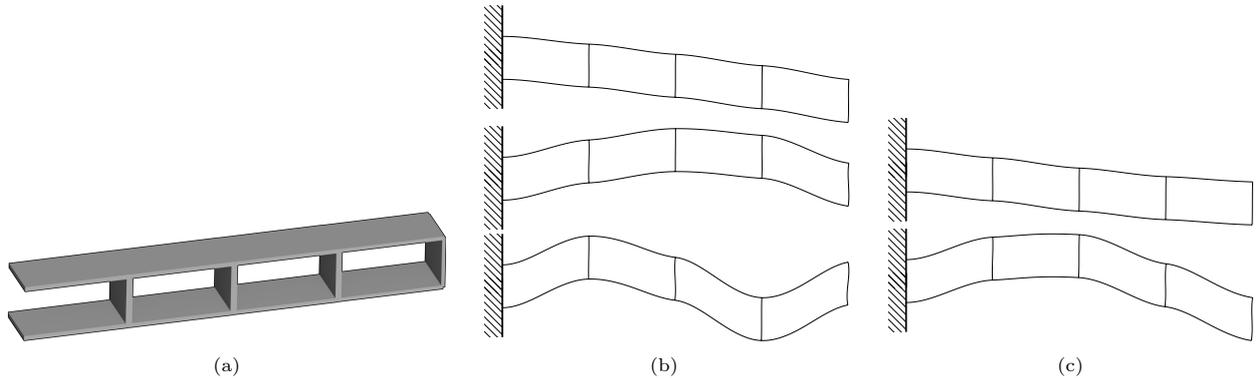
\begin{figure}[!t]
\centering
\begin{subfigure}{0.35\linewidth}
  \includegraphics[width=\linewidth]{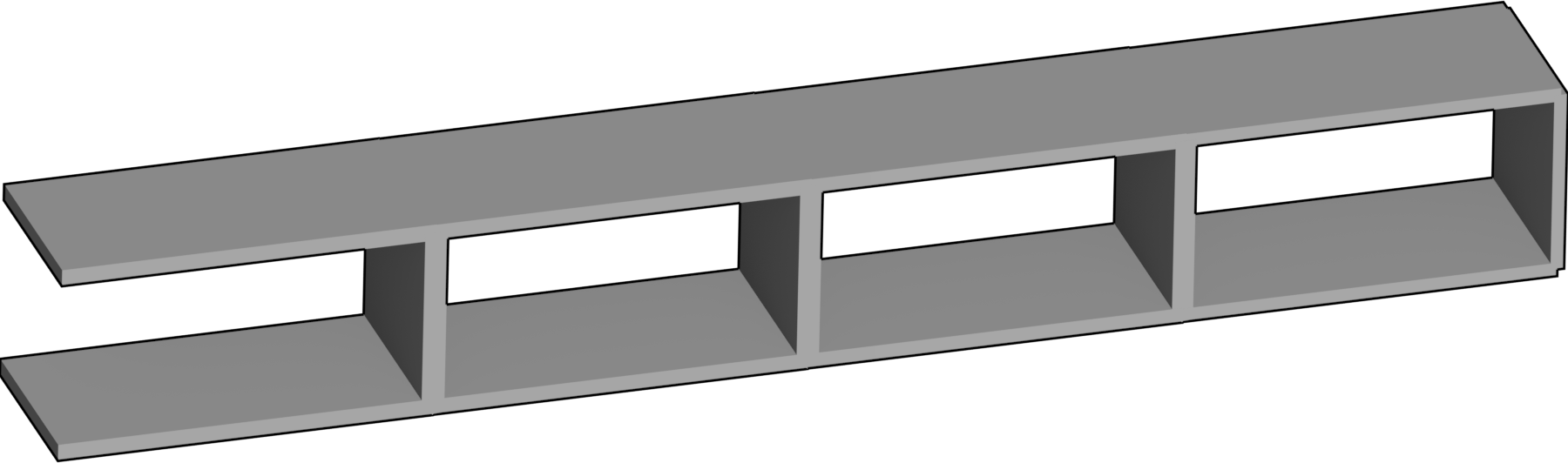}
  \caption{}
\end{subfigure}%
\hfill\begin{subfigure}{0.3\linewidth}
    \centering
    \scalebox{0.65}{
    \begin{tikzpicture}
    \begin{axis}[legend pos=north west, width=10cm, height=4.5cm, hide axis, unbounded coords=jump, unit vector ratio={1 1}]
    \addplot[mark=none] table {expframe.dat};
    \node (S) at (axis cs:0,0) {};
    \node (S2) at (axis cs:0,0.02070) {};
    \end{axis}
    \support{3}{S}[270];
    \support{3}{S2}[270];
    \end{tikzpicture}}
    \scalebox{0.65}{
    \begin{tikzpicture}
    \begin{axis}[legend pos=north west, width=10cm, height=4.5cm, hide axis, unbounded coords=jump, unit vector ratio={1 1}]
    \addplot[mark=none] table {expframe2.dat};
    \node (S) at (axis cs:0,0) {};
    \node (S2) at (axis cs:0,0.02070) {};
    \end{axis}
    \support{3}{S}[270];
    \support{3}{S2}[270];
    \end{tikzpicture}}
    \scalebox{0.65}{
    \begin{tikzpicture}
    \begin{axis}[legend pos=north west, width=10cm, height=4.5cm, hide axis, unbounded coords=jump, unit vector ratio={1 1}]
    \addplot[mark=none] table {expframe3.dat};
    \node (S) at (axis cs:0,0) {};
    \node (S2) at (axis cs:0,0.02070) {};
    \end{axis}
    \support{3}{S}[270];
    \support{3}{S2}[270];
    \end{tikzpicture}}
    \caption{}
  \end{subfigure}%
  \hfill\begin{subfigure}{0.3\linewidth}
    \scalebox{0.65}{
    \begin{tikzpicture}
    \begin{axis}[legend pos=north west, width=10cm, height=4.5cm, hide axis, unbounded coords=jump, unit vector ratio={1 1}]
    \addplot[mark=none] table {expframe1nm.dat};
    \node (S) at (axis cs:0,0) {};
    \node (S2) at (axis cs:0,0.02070) {};
    \end{axis}
    \support{3}{S}[270];
    \support{3}{S2}[270];
    \end{tikzpicture}}
    \scalebox{0.65}{
    \begin{tikzpicture}
    \begin{axis}[legend pos=north west, width=10cm, height=4.5cm, hide axis, unbounded coords=jump, unit vector ratio={1 1}]
    \addplot[mark=none] table {expframe2nm.dat};
    \node (S) at (axis cs:0,0) {};
    \node (S2) at (axis cs:0,0.02070) {};
    \end{axis}
    \support{3}{S}[270];
    \support{3}{S2}[270];
    \end{tikzpicture}}
    \caption{}
  \end{subfigure}
  \caption{Optimized design for the experimental frame problem: (a) optimized structure, (b) eigenmodes for the first three natural frequencies $30.0$~Hz, $126.7$~Hz, and $256.1$~Hz with the non-structural mass, and (c) eigenmodes for the first two natural frequencies $63.7$~Hz and $179.7$~Hz without the non-structural mass.}
  \label{fig:exp_frame_design}
\end{figure}

\subsubsection{Manufacturing and validation}

The optimized design was exported to a STEP file format and manufactured using the Bambu Lab H2D 3D printer with Pr\r{u}\v{s}a PETG filament. The slicer and printing parameters were set to ensure structural integrity and dimensional accuracy. The printed frame structure was subjected to two modal tests using three LSM6DS3TR-C accelerometers strategically placed at the centers of the bottom segments $2$--$4$ to capture the dynamic response (see Fig.~\ref{fig:exp_frame_test} for the experiment setup). In particular, we performed modal tests of the beam without the non-structural mass, and then with the non-structural mass of $20$~g placed at the bottom-right corner, see Fig.~\ref{fig:exp_frame_test}.

We excited the structure and recorded the acceleration responses sampled at 2.2 kHz for $4$~s. From the recorded data, we extracted the natural frequencies using Fourier transform techniques. For the case with the non-structural mass, the measured natural frequencies for the first three modes were consistently $30.33$~Hz, $133.23$~Hz, and $266.30$~Hz, which is by $1.1\%$, $5.2\%$, and $4.0\%$ higher than the predicted values. The beam exhibited very light damping, with damping ratios below $0.5\%$ estimated using logarithmic decrement. Thus, these values closely match the predicted frequencies from the finite element model, confirming the accuracy of the optimization and manufacturing process.

Without the non-structural mass, the finite element model predicts the first two eigenfrequencies as $63.74$~Hz and $179.69$~Hz with the eigenmodes depicted in Fig.~\ref{fig:exp_frame_design}c. In the corresponding modal test, the measured natural frequencies for the first two modes were consistently $65.32$~Hz and $185.81$~Hz, which is by $2.5\%$ and $3.4\%$ higher than the predicted values, again showing excellent agreement with the model.

\begin{figure}[!htbp]
  \centering
  \begin{subfigure}{0.2\linewidth}
    \includegraphics[width=0.95\linewidth]{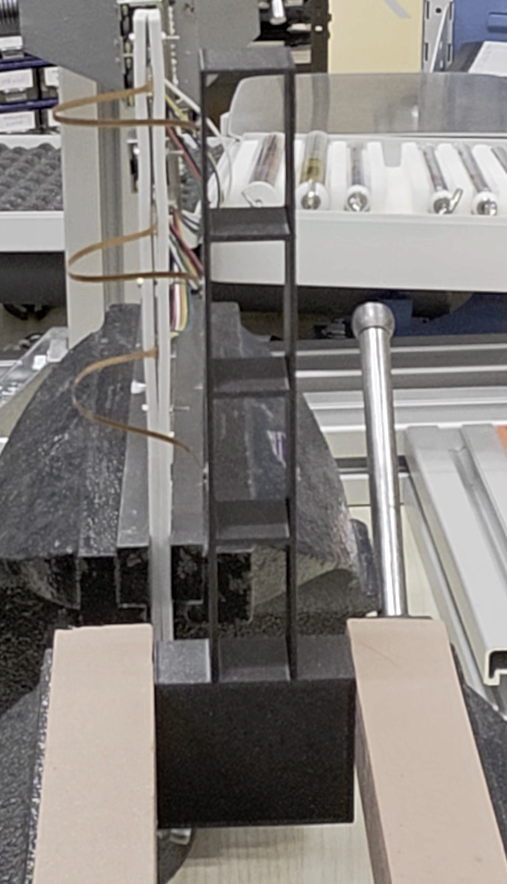}
    \caption{}
  \end{subfigure}%
  \begin{subfigure}{0.2\linewidth}
    \includegraphics[width=0.95\linewidth]{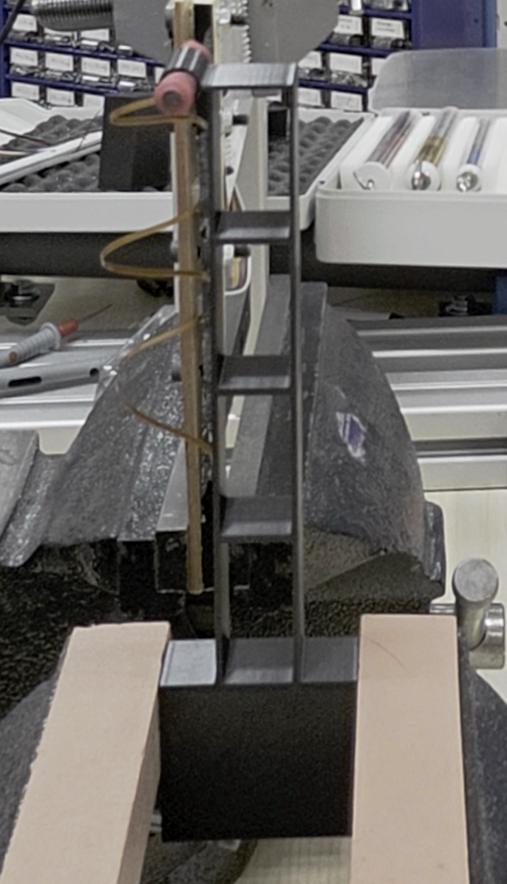}
    \caption{}
  \end{subfigure}
  \caption{Modal testing of the printed frame structure: (a) without non-structural mass, (b) with $20$~g mass at the bottom-right corner.}
  \label{fig:exp_frame_test}
\end{figure}

\section{Conclusions}\label{sec:conclusions}

In this paper, we studied minimum-weight design of undamped Euler--Bernoulli frame structures under single-frequency subresonant harmonic excitation, focusing on dynamic compliance and peak input power in both nominal and ellipsoidal-robust settings. Our main theoretical result is an exact semidefinite-programming reformulation of these response constraints as a free-vibration generalized eigenvalue condition with design-independent mass augmentation. This equivalence unifies static, dynamic, and modal constraints in a single framework and clarifies that the worst-case amplitudes align with the fundamental mode of the augmented eigenproblem. Although the resulting design problem is nonconvex as the structural stiffness depends polynomially on cross-sectional areas, the reformulation enables certified lower bounds via the moment-sum-of-squares hierarchy and construction of feasible upper bounds, yielding explicit optimality guarantees.

Numerical studies on three examples support the theory. For the $10$-segment frame of \citet[Section 5.1]{Ni2014}, we obtain designs with certified global optimality and numerically confirm equivalences among dynamic-compliance, peak-input-power, and free-vibration-eigenvalue formulations. For the larger $35$-segment tower of \citet[Section 5.2]{Ni2014}, we obtained high-quality local solutions together with rigorous bounds; the resulting designs are substantially lighter than the previously reported ones. In an additional case study we obtained a high-quality local solution, manufactured the structure, and experimentally validated the predicted frequency response, observing close agreement with the model.

The proposed pipeline offers three benefits: (i) unified modeling handle through free-vibration eigenvalue constraints which covers static, dynamic-compliance, and peak-input-power constraints; (ii) robustness to load uncertainty is directly embedded in the semidefinite programming machinery; and (iii) certification---often missing in nonconvex topology optimization---is available through convergent relaxations.

The present theory targeted single-frequency, uniform-phase loads below the first resonance without damping; scalability of generic SDP/mSOS relaxations may become costly for large ground-structures. Extensions to deal with these aspects will be considered in future work.

\section*{Acknowledgments}
This work was funded by the Czech Science Foundation (project No. 22-15524S) and, from January 2025, co-funded by the European Union under the ROBOPROX project (reg. no. CZ.02.01.01/00/22 008/000459). The authors would like to thank Pavel Rychnovský for his assistance with the experimental validation.

\section*{Declaration of generative AI and AI-assisted technologies in the writing process}
During the preparation of this work the authors used the ChatGPT service in order to improve the clarity and quality of the writing. After using this tool, the authors reviewed and edited the content as needed and take full responsibility for the content of the published article.

\bibliography{liter.bib}
\bibliographystyle{abbrvnat}

\end{document}